\documentclass[11pt]{amsart}

\usepackage{amssymb}

\usepackage{graphicx}

\usepackage{amsmath}

\usepackage{verbatim}

\makeindex

\makeatletter
\@namedef{subclassname@2020}{\textup{2020} Mathematics Subject Classification}
\makeatother

\input xy
\xyoption{all}

\newtheorem{theorem}{Theorem}[section]

\newtheorem{construction}[theorem]{Construction}

\newtheorem{corollary}[theorem]{Corollary}

\newtheorem{definition}[theorem]{Definition}

\newtheorem{question}[theorem]{Question}

\newtheorem{example}[theorem]{Example}

\newtheorem{lemma}[theorem]{Lemma}

\newtheorem{proposition}[theorem]{Proposition}
\newtheorem{remark}[theorem]{Remark}

\newcommand{\prskip}{\vspace{8pt}} 
\newcommand{\thmskip}{\vspace{10pt}} 
\newcommand{\pfskip}{\vspace{6pt}} 
\newcommand{\sectskip}{\vspace{50pt}} 
\newcommand{\introskip}{\vspace{25pt}} 
\newcommand{\Z}{\mathbb{Z}}

\newcommand\Monu{\mathbf{Mon}_1}
\newcommand\Zpos{\Z_{>0}}
\newcommand\vhat{\widehat{v}}
\newcommand\Vhat{\widehat{V}}

\DeclareMathOperator\Hom{Hom}

\begin{document}

\title{Regular Ring Properties Degraded \\ Through Inverse Limits}  \vspace{5mm}

\author{Pere Ara}
\address{Pere Ara,
        Departament de Matem\`{a}tiques,
        Universitat Aut\`{o}noma de Barcelona,
        08193 Bellaterra, Barcelona, Spain, and
        Centre de Recerca Matem\`atica, Edifici Cc, Campus de Bellaterra, 08193 Cerdanyola del Vall\`es, Barcelona, Spain}
\email{pere.ara@uab.cat}

\author{\ Ken  Goodearl  }
\address{Department of Mathematics, University of California, Santa Barbara, CA 93106, USA}
\email{goodearl@math.ucsb.edu}

\author{\ Kevin C. O'Meara}
\address{University of Canterbury, Christchurch, New Zealand}
\email{staf198@uclive.ac.nz}

\author{\ Enrique Pardo}
\address{Departamento de Matem\'aticas, Facultad de Ciencias, Universidad de C\'adiz, Campus de Puerto Real, 11510 Puerto Real (C\'adiz), Spain}
\email{enrique.pardo@uca.es}

\author{\ Francesc Perera}
\address{
F.~Perera,
Departament de Matem\`{a}tiques,
Universitat Aut\`{o}noma de Barcelona,
08193 Bellaterra, Barcelona, Spain, and
Centre de Recerca Matem\`atica, Edifici Cc, Campus de Bellaterra,  08193 Cerdanyola del Vall\`es, Barcelona, Spain}
\email[]{francesc.perera@uab.cat}
\urladdr{https://mat.uab.cat/web/perera}



\vspace{10mm}

\begin{abstract}
We give a number of constructions where inverse limits seriously degrade properties of regular rings, such as unit-regularity, diagonalisation of matrices, and finite stable rank. This raises the possibility of using inverse limits to answer the long standing Separativity Problem (in the negative).
\end{abstract}

\maketitle

The motivation for this work is to explore inverse limits as a new tool to settle in the negative the Separativity Problem (SP) for (von Neumann) regular rings. This problem, which was posed by Ara, Goodearl, O'Meara, and Pardo in 1994, asks if all regular rings (or exchange rings) $R$ satisfy the property
\[
    A \ \oplus \ A \  \cong \  A \ \oplus \ B \ \cong \ B \ \oplus \ B \ \Longrightarrow \ A \ \cong \ B
\]
for finitely generated (f.g.) projective $R$-modules $A, B$. No counter constructions have worked to date, so why not try a new one! While a resolution of the SP using this new tool has so far eluded us, the constructions and results in this paper do confirm that inverse limits can seriously degrade regular ring properties. For instance, building on a construction of Bergman in the 1970s, and modified by O'Meara in 2017, we construct an inverse limit of unit-regular rings which remains regular but is no longer unit-regular. Thus we have degraded cancellation of f.g.\ projectives but without destroying regularity. However, these degradations do require certain restrictions on the connecting maps $f_i: R_{i+1} \rightarrow R_i$ in $\varprojlim R_i$, such as not being surjective. All this gives added urgency to the question of whether \emph{an inverse limit can also degrade the property of separativity?}
\prskip

Recall a ring $R$ is \emph{regular} if each $a \in R$ has an inner inverse $b$ in the sense that $a = aba$. If $b$ can always be chosen to be a unit, then $R$ is \emph{unit-regular}. And $R$ is an \emph{exchange} ring if its f.g.\ projective modules possess the finite exchange property. Exchange rings are more general than regular rings (for instance they include all locally finite-dimensional algebras) but share many properties with regular rings, such as an abundance of idempotents and, most notably, the refinement property for direct sums of f.g.\ projective modules. For simplicity, however, we restrict our discussions to mostly regular rings even though much of what we do can also be done for exchange rings.
\prskip

Here is a quick outline of the paper. Section 1 reminds us of where separativity fits within regular rings: it is a much more general, unifying notion than first appears, and the resolution of the SP has important ramifications. In Section 2, we recall the basics of inverse limits, while in Section 3 we place our approach in the framework of varieties in the universal algebra sense. This is a powerful point of view because of the ability to take free objects in a variety. Constructions illustrating how inverse limits can degrade properties are given in Section 4, along with some positive results in Section 5 when the connecting maps are surjective. In Section 6, we relate inverse limits of regular rings $R$ to inverse limits of their monoids $V(R)$ of f.g.\ projectives. A most instructive example is given in Section 7 of a regular inverse limit $R$ of regular rings $R_i$ where, despite the connecting maps being surjective, $V(R)$ is not isomorphic to $\varprojlim V(R_i)$. The construction involves graph algebras, and appeals to some nontrivial results within that area.  Finally, in Section 8, we examine an intermediate step in constructing a non-separative regular ring.
\thmskip

\section{Separativity}

The notion itself stems from semigroup theory in the 1950s. For regular rings, we view separativity in an equivalent form (to the definition in the Introduction) of a broad cancellation property for f.g.\ projective modules, akin to those associated with finite stable rank:
\[
A \ \oplus \ C \  \cong \  B \ \oplus \ C \ \Longrightarrow \ A \ \cong \ B
\]
for f.g.\ projective $R$-modules $A,B,C$ when $C$ is isomorphic to both a direct summand of a finite direct sum of copies of $A$ and of a finite direct sum of copies of $B$. See \cite[Lemma 2.1]{AGOP1}. (We can never expect universal cancellation because not all regular rings are unit-regular.) Another interesting equivalent view of separativity is that ``multi-isomorphism'' of f.g.\ projective modules coincides with isomorphism:
\[
   A^n \ \cong \ B^n \ \  \forall n > 1 \ \Longrightarrow \ A \ \cong \ B.
\]

There were three major theorems established in the 1990s concerning separative regular rings:
\prskip

\noindent $\bullet$ \textbf{Extension Theorem.}  Separativity for regular (or exchange) rings is preserved in extensions of ideals $I$ (as non-unital rings) by factor rings $R/I$: $R$ is separative iff $I$ and $R/I$ are separative.
\prskip

\noindent $\bullet$ \textbf{Diagonalisation Theorem.}  Square matrices $A$ over separative regular rings are equivalent to diagonal matrices: $PAQ = D$ for some invertible $P, Q$ and diagonal $D$.
\prskip

\noindent $\bullet$ \textbf{GE Theorem.}  Invertible matrices over separative regular rings are products of elementary matrices.
\prskip

\noindent Moreover, among regular rings $R$, the separative ones are characterised by the property that $2 \times 2$ matrices over corner rings $eRe$ (where $e$ is an idempotent) can be diagonalised by elementary matrices. See  \cite{AGOP1}, \cite{AGOP2}, and \cite{AGOR} for more detail.
\prskip

It is the Extension Theorem that makes life difficult for us in directly constructing non-separative regular rings. However, this does not appear to be an impeding factor in inverse limit constructions, because regularity itself is closed under extensions whereas we construct many non-regular inverse limits of regular rings.
\thmskip

\section{The basics of inverse limits}\label{Section:InverseLimits}
\thmskip

Inverse limits are the dual of direct limits. However, inverse limits can present a greater challenge to our intuition in predicting what a particular limit might look like! Roughly speaking, just as we can think of a direct limit as a \emph{``fancy type of union''}, we can think of an inverse limit as a \emph{``fancy type of intersection''}. Two excellent references for these limits, as well as universal algebra generally, are George Bergman's book \cite[Chapter 9]{GBbook}  and Nathan Jacobson's book \cite[Chapter 2]{BasicAlgebra}.
\prskip

For simplicity, we will restrict ourselves to the indexing set $\mathbb{Z}_{>0}$. But we certainly don't rule out much more complex (even uncountable) directed sets playing a critical role. Given objects $A_i$ in some category, and  connecting morphisms $f_i\colon A_{i+1}\to A_i$, the \emph{inverse limit} $\varprojlim A_i$ is an object $L$ in the category, together with morphisms $\pi_i\colon L\to A_i$, such that $f_i\circ\pi_{i+1}=\pi_i$, which satisfy the universal property shown in the picture:

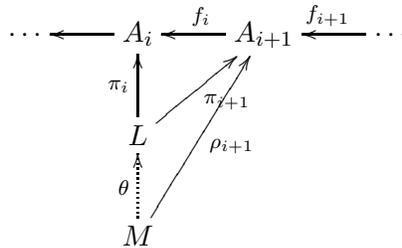
\[
\xymatrix{\cdots & A_i \ar[l] & A_{i+1} \ar[l]_{f_i} & \cdots \ar[l]_{f_{i+1}}  \\
     & L \ar[u]^{\pi_i} \ar@{>}[ru]_{\pi_{i+1}}   \\
     & M \ar@{.>}[u]^{\theta} \ar@{>}[ruu]_{\rho_{i+1}} }
\]

\noindent where $\rho_i: M \rightarrow A_i$ also satisfy $f_i \circ\rho_{i+1} = \rho_i$, and the map $\theta$ is unique.

In a general (universal) algebra, we can take
\[
\varprojlim A_i := \biggl\{(a_i)_{i\in \Z_{>0}}\in \prod_{i\in \Z_{>0}}A_i\, :\, f_i(a_{i+1})=a_i \text{ for each }i\in \Z_{>0}\biggr\}
\]
and take $\pi_i((a_i)_{i\in \Z_{>0}})=a_i$. This set might be empty, but it is automatically nonempty if there is some 0-ary operation on the algebras.
\prskip

The prototype of an inverse limit $\varprojlim R_i$ in the category of rings is the intersection of a descending chain $R_1 \ \supseteq  \ R_2 \ \supseteq \  R_3 \ \supseteq \ \ldots$ of subrings, with the connecting maps the inclusion maps. Moreover, in a sense to be made clear later (see Proposition \ref{P:Inj}), all inverse limits of rings can be viewed this way.
\prskip

Any direct product $\prod_{i = 1}^{\infty}\,R_i$ of rings $R_i$ can be viewed as an inverse limit $\varprojlim\,S_i$ of the rings $S_i = R_1 \times R_2 \times \ \cdots  \times\ R_i$ with connecting maps $f_i: S_{i+1} \rightarrow S_i$ the natural projections.
\prskip

The classical, and most instructive, example of an inverse limit is the ring of $p$-adic integers constructed from the rings $\mathbb{Z}/p^i$ using the natural connecting maps. This is an example of an inverse limit of exchange rings with surjective connecting maps and the inverse limit is also an exchange ring (being a local ring). So that is encouraging.
\sectskip


\section{A variety setting}

\introskip
The use of certain varieties (in the universal algebra sense) as a general framework for our inverse limits proves most useful. Among other things it allows us access to free objects, such as a free separative regular ring on given generators. Recall, a variety is determined by a class of algebraic objects, a set of operations, and universal identities satisfied under those operations. Thus we have the variety \textbf{Ring} of rings with unity relative to the operations $+, -, \cdot, 0, 1$. So two binary operations, one unary, and two nullary (constants). And satisfying the usual identities such as the distributive law $a\cdot(b + c) = a \cdot b + a \cdot c$ for all $a, b, c \in R$. By throwing in the unary operation $'$ with identity $a = aa'a$ we get the variety $\textbf{Reg}$ of regular rings. To get the variety \textbf{UnitReg} of unit-regular rings, we can impose the extra two identities $a'(a')' = 1$ and $(a')'a'= 1$ on \textbf{Reg}.  Two other important varieties for us are \textbf{DiagReg} and \textbf{SepReg}. The former encapsulates diagonalisation of $2 \times 2$ matrices {\small{$\left[\begin{array}{cc}
            w & x \\
            y & z
            \end{array}\right]$ }}
over a regular ring $R$ by elementary matrices (3 each side):
\[
 \left[\begin{array}{cc}
            1 & c \\
            0 & 1
            \end{array}\right]
\left[\begin{array}{cc}
            1 & 0 \\
            b & 1
            \end{array}\right]
\left[\begin{array}{cc}
            1 & a \\
            0 & 1
            \end{array}\right]
\left[\begin{array}{cc}
            w & x \\
            y & z
            \end{array}\right]
 \left[\begin{array}{cc}
            1 & d \\
            0 & 1
            \end{array}\right]
\left[\begin{array}{cc}
            1 & 0 \\
            e & 1
            \end{array}\right]
\left[\begin{array}{cc}
            1 & f \\
            0 & 1
            \end{array}\right] \\
\]
\[                         = \ \
\left[\begin{array}{cc}
            \ast & 0\\
            0 & \ast
            \end{array}\right].
\]
So given a 4-tuple $(w,x,y,z)$ of elements of $R$, we can pick out the entries $a, b, c, d, e, f$ in the elementary matrices in terms of six quaternary-operations, and then specify that the $(1,2)$ and $(2,1)$ entries of the product are zero using the obvious multilinear equations. Of course, we still need the unary operation $'$ of inner inverse to ensure $R$ is regular but notice \emph{the diagonalisation itself does not involve the inner inverse $'$}. To describe the variety \textbf{SepReg} of separative regular rings in simple terms (at least conceptually), we use the characterisation mentioned before of separative regular rings in terms of diagonalising $2 \times 2$ matrices over corner rings. But with a nice recent improvement in \cite{AGNPP}: only 3 elementary matrices are needed on each side in the diagonalisation, and nothing less on each side will work in general. Fortunately also, we can parametrise idempotents as  elements of the form $vv' + v(1-vv')$ as $v$ ranges over the members of $R$. Again see \cite{AGNPP}. So combined with our description of \textbf{DiagReg}, this makes $\textbf{SepReg}$ a variety. Notice that this viewpoint \emph{does not prioritise one inner inverse operation over another}. As shown in \cite{AGNPP} it is possible also to view the class of exchange rings as a variety but this is much more delicate.
\prskip

A most surprising result obtained in \cite[Theorem 4.6]{AGNPP} is that it is possible to view \textbf{SepReg} as a \emph{subvariety} of $\textbf{Reg}$ by keeping the same signature of being a ring with a particular inner inverse operation $'$, but imposing an additional equation involving four recursive (layered) applications of $'$ within a quaternary regular ring expression to get the necessary diagonal reduction. However, this approach does require some nontrivial technicalities, and the additional defining identity is not exactly obvious.
\prskip

When we talk about an inverse limit in a given variety, we assume the connecting maps preserve \emph{all} the operations of the variety, in which case the limit also resides in that variety. So in the case of an inverse system $(R_i,f_i)$ in \textbf{Reg}, each $R_i$ is equipped with a particular inner inverse operation $'$, and the connecting maps $f_i: R_{i+1} \rightarrow R_i$ are ring homomorphisms satisfying $f_i(a') = (f_i(a))'$ for all $a \in R_{i+1}$.  Clearly, if we work in \textbf{SepReg}, then the resulting inverse limit will also be in {\bf SepReg}.  However, if we are in \textbf{Ring} or \textbf{Reg}, then there is some hope that the inverse limit might fail to be separative.
\thmskip


\section{Constructions where properties are degraded}

\thmskip

One of the earliest examples of how inverse limits can degrade properties of a regular ring was given by George Bergman in the 1970s and recorded in \cite[Example 1.10]{VNRR}.
\thmskip

\begin{construction}\label{Constr:not reg}
An inverse limit of regular rings taken in \emph{\textbf{Ring}} that is not regular.
\end{construction}
\prskip

\noindent Let $R$ be a ring, and let $S$ be the subring of $R^{\Z_{>0}}$ consisting of all sequences that eventually stabilize:
\[
S:=\{(x_i)_{i\in \Z_{>0}}\in R^{\Z_{>0}} \, :\, x_i = x_{i+1} = \cdots \ \mbox{for some}  \ i > 0\}.
\]
Let $\varphi$ be an automorphism of $R$.  The sequences in $S$ are restricted in their tails, but using $\varphi$ we can restrict these sequences from the start.  Namely, we can pass to the subring
\[
S_n :=\{(x_i)_{i\in \Z_{>0}}\in S\, :\, \varphi(x_i)=x_{i+1}\ \text{if $i\leq n$}\}.
\]
In this way we obtain a sequence of subrings, $S\supseteq S_1 \supseteq S_2 \supseteq \ldots$, each isomorphic to $S$, and the inverse limit $\varprojlim S_i$ of these subrings, via the inclusion maps, is isomorphic to the subring of $R$ fixed by $\varphi$.
\prskip

Now specialise the construction to $R = M_2(F)$ for a fixed field $F$ and the inner automorphism $\varphi$ of conjugation by
\[
C \ = \ \left[\begin{array}{cc}
1 & 1\\
0 & 1\end{array}\right].
\]
Inasmuch as the centraliser $T$ of $C$ in $R$ is
 \[
T \ = \ \left\{\left[\begin{array}{cc}
a & b\\
 0 & a\end{array}\right]\, :\, a,b\in F\right\}\cong F[x]/\langle x\rangle^2,
\]
which has a nonzero nilpotent ideal, $T$ cannot be regular. Therefore $\varprojlim S_i$ is not regular. \hfill $\square$
\thmskip

In the above construction, the connecting maps (inclusions) cannot preserve fixed inner inverse operations on each $R_i$. Otherwise, the limit takes place in the variety \textbf{Reg} and therefore would be regular!
\prskip

This type of construction can be generalized.  Instead of using one automorphism, we could use up to countably many, as follows.  Fix, once and for all, a map $\sigma\colon \Z_{>0}\to \Z_{>0}$ with the property that each positive integer occurs as an output infinitely often.  (For instance, $\sigma$ could send a number to the sum of its binary digits.)  Let $\theta_1,\theta_2,\ldots$ be a sequence (either finite or countable) of automorphisms of $R$.  We then define subrings of $S$, by setting
\[
S_n :=\{(x_i)_{i\in \Z_{>0}}\in S\, :\, \theta_{k}(x_{i})=x_{i+1}\ \text{if $i\leq n$ and $\sigma(i)=k$}\}.
\]
Now, the inverse limit is (isomorphic to) the subring of $R$ fixed by all of the $\theta_k$.
\prskip

\begin{construction}
There exists an inverse limit of unit-regular rings, taken in \textup{\bf Ring}, with the property that some $2\times 2$ matrix over the ring is not diagonalizable.
\end{construction}

Consider the ring $R:=M_{4}(F)$ with $F$ a field.  Let $\theta_1$, $\theta_2$, and $\theta_3$ be conjugations by the respective elementary matrices
{\small{\[
 I_4+e_{13} = \left[\begin{array}{cccc}
1 & 0 & 1 & 0 \\
  & 1 & 0 & 0 \\
  &   & 1 & 0\\
  &   &   & 1
\end{array}\right],  \ \
I_4+e_{14} = \left[\begin{array}{cccc}
1 & 0 & 0 & 1\\
  & 1 & 0 & 0\\
  &   & 1 & 0\\
  &   &   & 1
\end{array}\right], \  \ \
I_4+e_{23} = \left[\begin{array}{cccc}
1 & 0 & 0 & 0\\
  & 1 & 1 & 0\\
  &   & 1 & 0\\
  &   &   & 1
\end{array}\right].
\]}}
\noindent The subring of $R$ fixed by these three automorphisms is
\[
\left\{\left[\begin{array}{cc|cc}
a & 0 & b & c\\
  & a & d & e\\ \hline
  &   & a & 0\\
  &   &   & a\\
\end{array}\right]\, :\, a,b,c,d,e\in F \right\}\cong F[x_1,x_2,x_3,x_4]/\langle x_1,x_2,x_3,x_4\rangle^2.
\]
By \cite[Example 2.12]{AGOR}, there are $2\times 2$ matrices over this ring that cannot be diagonalized. The ring $R$ occurs as an inverse limit of unit-regular rings, in the variety {\bf Ring}, for the same reason as given in Construction \ref{Constr:not reg}. \hfill $\square$
\prskip

In the previous two constructions, the automorphisms we employed were conjugations.  In our next construction we will require an outer automorphism.
\prskip

\begin{construction} \label{Constr:not unit regular}
There is an inverse limit of unit-regular rings, taken in \textup{\bf Reg}, that is not unit-regular \textup{(}but is automatically regular\textup{)}.
\end{construction}

This construction is also based on an earlier classical one of Bergman in the 1970s and recorded in \cite[Example 5.12]{VNRR}. It exhibits a regular subalgebra $R$ of a unit-regular algebra $Q$ (over an arbitrary field $F$) but with $R$ not unit-regular. It is not possible to realize $Q$ inside a countably-infinite matrix ring over $F$ if all the matrices are column-finite or all row-finite. However, in \cite{O2} it was shown how to realise $Q$ and $R$ inside an $\omega \times \omega$ matrix ring when viewed as a Morita context. We briefly recall the details.
\prskip

Fix a field $F$. Let $B$ be the formal $2 \times 2$ matrix ring
\[
  B \ = \ \left[\begin{array}{cc}
    U & M \\
    N & V
    \end{array}\right]
\]
where $U$ is the ring of all $\omega \times \omega$ row-finite matrices over $F$, \ $M$ is the $F$-space of all  $\omega \times \omega$ matrices over $F$, \ $V$ is the ring of all $\omega \times \omega$ column-finite matrices over $F$, and $N$ is the space of all $\omega \times \omega$    matrices with only finite many nonzero entries. Establishing that $B$ is actually a ring is a delicate matter (!). See \cite{O2}. The key to the construction is that there is a natural copy $K$ of the Laurent series ring $F[t,t^{-1}]$ inside $B$ given by
\[
     b_mt^{-m} + \cdots + b_1t^{-1} + c_0 + c_1t + c_2t^2 + \cdots + c_nt^n + \cdots
\]
\ \ \ $\longmapsto$
{\small{\[
                       \left[\begin{array}{c c c c c c c|c c c c c c c}
                            c_0 & b_1 & b_2 &\cdots & b_m &0&\cdots     & c_0 & c_1 & c_2 & \cdots & c_n & \cdots &\\
                            c_1 & c_0 & b_1 & b_2  & \cdots &b_m &\ddots& c_1 & c_2 & \cdots & c_n  & \cdots & &\\
                            c_2 & c_1 & c_0 & b_1 & b_2 &\cdots& \ddots & c_2 & \cdots & c_n & \cdots &  &   &\\
                           \vdots &\ddots&\ddots&\ddots&\ddots&\ddots&  &\vdots  &  &  &  &  &  & \\
                           c_n    &  &  &  &  &          &              &c_n  &  &  &  &  &  & \\
                           \vdots & c_n &  &  &  &         &            &\vdots  &  &  &  &  &  & \\
                             &     &\ddots     &  &  &  &               &  &  &  &  &  &  & \\  \hline
                           b_0 & b_1 & b_2 & \cdots & b_m & 0 & \cdots & c_0 & c_1 & c_2 &\cdots&c_n    &\cdots&  \\
                           b_1 & b_2 & \cdots & b_m & 0 & \cdots&      & b_1 & c_0 & c_1 & c_2 &\cdots& c_n &\cdots\\
                           b_2 & \cdots & b_m & 0 & \cdots & &         & b_2 & b_1 & c_0 & c_1 & c_2 &\cdots&\ddots \\
                          \vdots&  &  &  &  & &                        &\vdots&\ddots&\ddots&\ddots&\ddots&\ddots   \\
                          b_m  &  &  &  &  & &                         & b_m &  &  &  &  & &  \\
                          0  &   &  &  &  & &                          & 0   & b_m &  &  &  &  & \\
                          \vdots & & & & &  &                          &\vdots&\ddots&\ddots& & & &\\
                             \end{array}\right].
\]}}
Also $B$ is a prime ring with nonzero socle
\[
  P \ = \ \left[\begin{array}{cc}
    I & IM+MJ \\
    N & J
    \end{array}\right]
\]
where $I$ and $J$ are the socles of $U$ and $V$ respectively. Let $\pi: B \rightarrow B$ be the projection onto the diagonal
\[
   \left[\begin{array}{cc}
    u & m \\
    n & v
    \end{array}\right]
 \ \mapsto \ \\
 \left[\begin{array}{cc}
    u & 0 \\
    0 & v
    \end{array}\right].
\]
This is a ring homomorphism modulo $P$.
We set
  \[
  R \ = \ \left[\begin{array}{cc}
    I & 0 \\
    0 & J
    \end{array}\right]  +  \pi(K), \ \
  Q \ = \ R \, + \, P \ =\,P \, + \, \pi(K).
\]
As shown in \cite{O2}, the ring $R$ is regular, but not unit-regular, while the ring $Q$ is unit-regular.
\prskip

For the remainder of this construction, assume that $F\neq \mathbb{Z}/2$, and fix $\alpha\in F-\{0,1\}$.  Let $\theta$ be the automorphism of $B$ induced by conjugation by the matrix {\Small{$\left[\begin{array}{cc} \alpha & 0\\ 0 & 1\end{array}\right]$}}.  It is easy to see that $\theta$ restricts to an outer automorphism of $Q$, whose fixed subring is $R$. By applying the technique in Construction 4.1, we see $R$ occurs as an inverse limit of unit-regular rings in {\bf Ring}.  Our next lemma shows that the inverse limit may be forced to occur in {\bf Reg}. \hfill $\square$
\prskip

\begin{lemma}\label{L:Inversreg}
Assume that $(R_i,f_i)_{i\in \Z_{>0}}$ is an inverse system in the variety {\bf Ring}, where each homomorphism $f_i$ is injective, and where each ring $R_i$ is regular.  If the inverse limit in \textup{\bf Ring} happens to be regular, then we can define inner inverse operations on each of the rings $R_i$ so that the inverse limit occurs in \textup{\bf Reg}.
\end{lemma}

\begin{proof}
Identify the inverse limit with
\[
R:=\biggl\{\underline{x}=(x_{i})_{i\in \Z_{>0}}\in \prod_{i\in \Z_{>0}}R_i\,:\, f_i(x_{i+1})=x_i\ \text{for all $i\in \Z_{>0}$}\biggr\}
\]
and let $\pi_i\colon R\to R_i$ be the natural projection. Let $'$ be an inner inverse operation on $R$. Since the $f_i$ are injective, either
\begin{itemize}
\item[(1)] $x_i$ occurs as the $i$th coordinate of exactly one infinite sequence $\underline{x}\in R$, or
\item[(2)] $x_i$ occurs as the $i$th coordinate of exactly one finite sequence $(x_1,\ldots x_n)$, for some $n\geq i$, with the property that $f_j(x_{j+1})=x_{j}$ for each $1\leq j\leq n-1$, and where $x_n\notin im(f_n)$.
\end{itemize}
In the latter case, we will call $x_n$ the \emph{ultimate ancestor} of $x_i$. For convenience we will let $x_i^{\ast}$ denote some fixed choice of an inner inverse of $x_i \in R_i$. We now define a new inner inverse operation $'$ on each $R_i$ in terms of $^{\ast}$ and the operation $'$ on $R$ by the rule:

\[
x_i': \ = \ \
\begin{cases}
\pi_{i}(\underline{x}')   \ \ \ \  &\text{if $x_i=\pi_i(\underline{x})$ for some $\underline{x}\in R$,}  \\
 (f_i\circ \cdots \circ f_{n-1})(x_n^{\ast}) &\mbox{if $x_n$ is the ultimate ancestor of $x_i$.}  \\
\end{cases}
\]
These are well-defined inner inverses on each $R_i$ and respected by the connecting maps.
\end{proof}
\prskip

\begin{remark}
The conjugation matrix, $C:=\left[\begin{smallmatrix}\alpha & 0\\ 0 & 1\end{smallmatrix} \right]$, used in Construction 4.3 is not a member of the ring $Q$, but of the bigger ring $B$.  It is tempting to arrange that this matrix is in $Q$, to force $\theta$ to be an inner automorphism.  Equivalently, we would want the idempotent
\[
e:=\begin{bmatrix}
1 & 0\\ 0 & 0
\end{bmatrix}=(\alpha-1)^{-1}(C-1_Q)\in Q.
\]
However, then we would lose the unit-regularity of $Q$, because the corner ring $eQe$ would not even be directly finite.  This shows the importance of having a bigger universe, $B$, in which to work. \hfill $\square$
\end{remark}

\begin{remark}
Any example of an inverse limit $R =\varprojlim R_i$ of unit-regular rings that is regular but not unit-regular has to be another example of the Bergman-type (that is, provides a regular subring $R$ of a unit-regular ring $Q$ with $R$ not unit-regular) because $R$ sits inside $\prod R_i$, which is a product of unit-regular rings, whence unit-regular. So constructions such as 4.3 are never easy.
\end{remark}
\prskip

\begin{construction}
There is an inverse limit in \emph{\textbf{Reg}}, of a sequence of regular rings with stable rank $2$, that is regular but has infinite stable rank.
\end{construction}

Let $F$ be a field with more than two elements and fix $\alpha\in F-\{0,1\}$. By a result of Menal and Moncasi in \cite[Example 3, p.\ 38]{MeMo} there is a regular $F$-algebra $R$ of stable rank $2$ that has a corner $eRe$ with infinite stable rank.  The fixed ring $T$ of conjugation of $R$ by $\alpha e + (1-e)$ is $T = eRe + (1-e)R(1-e)$, which has infinite stable rank. We can now apply the same argument used in Construction \ref{Constr:not reg} to obtain our inverse limit once we show the ring $S$ of sequences of elements of $R$ that stabilize is also of stable rank 2. Given any three elements $(a_i),(b_i),(c_i)\in S$ that are right unimodular in $S$, then for each $i\in \Z_{>0}$ we can fix elements $x_i,y_i\in R$ such that $u_i :=a_i+b_ix_i+c_iy_i\in U(R)$.  Moreover, as the $a_i,b_i,c_i$ eventually become constant, we can assume that $x_i$ and $y_i$ likewise become constant, so that $(x_i),(y_i)\in S$.  Moreover,
\[
(a_i)+(b_i)(x_i)+(c_i)(y_i)\in U(S),
\]
hence $S$ has stable rank at most 2. On the other hand,  $S$ has the ring $R$, whose stable rank is 2, as a homomorphic image, so $S$ must have stable rank 2. \hfill $\square$.

\begin{remark}
The only possibilities for the stable rank of separative exchange rings are $1$, $2$, or $\infty$, by \cite[Theorem 3.3(a)]{AGOP1}.
\end{remark}

It turns out that, in quite general settings, the injectivity hypothesis on the connecting maps in an inverse system is no limitation on the types of inverse limits that can arise.  The following proposition makes this formal.

\begin{proposition}\label{P:Inj}
Given any inverse system $(S_i,f_i)_{i\in \Z_{>0}}$ in {\bf Ring}, then there is another inverse system $(T_i,g_i)_{i\in \Z_{>0}}$ with each $g_i$ injective, and with an isomorphic inverse limit. Moreover, if $P$ is an isomorphism-invariant property of rings satisfied by each $S_i$, which passes to countable direct products, then each $T_i$ has $P$.
\end{proposition}

\begin{proof}
For each $j\in \Z_{>0}$, let
\[
T_j:=\{\underline{x}=(x_i)_{i\in \Z_{>0}}\in \prod_{i\in \Z_{>0}}S_i \, :\, f_{i}(x_{i+1})=x_i\ \text{for each positive integer $i<j$}\}.
\]
The inclusions $T_{j+1}\subseteq T_j$ induce a new inverse system.  Moreover, the intersection
\[
\bigcap_{i\in \Z_{>0}}T_i=\{\underline{x}=(x_i)_{i\in \Z_{>0}}\in \prod_{i\in \Z_{>0}}S_i \, :\, f_{i}(x_{i+1})=x_i\ \text{for each $i\in \Z_{>0}$}\}
\]
is an inverse limit of this new system as well as the original inverse system. Notice that $T_j\cong \prod_{i\ge j}S_i$, so the last claim quickly follows.
\end{proof}

\begin{remark}
The idea used in Proposition \ref{P:Inj} applies to more general inverse systems and appears in the proof of \cite[Theorem 2]{BergmanInjectives}.
\end{remark}
\thmskip


\section{Some positive results}\label{S:positive}
\thmskip

When the connecting maps in inverse limits are surjective, the behaviour of the limit is much better. We will make use of the fact that for an ideal $I$ of a regular ring $R$, inner inverses of an element $\bar{a} \in R/I$ lift to inner inverses of $a$. This is a well-known consequence of a lemma of Brown and McCoy \cite[Lemma 1]{BM} (see also \cite[Lemma 1.3]{VNRR}). Also if $R$ is unit-regular, units lift modulo an ideal (see \cite[Lemma 3.5]{Bac}). We need a stronger form of the latter result:
\prskip

\begin{lemma}\label{L:Lifting}
Let $R$ be a ring and $I$ an ideal of $R$. Assume that for any idempotent $e \in R$, all units of $\overline{e} \overline{R} \overline{e}$ lift to units of $eRe$. (This holds in case $R$ is unit-regular.) Let $a \in R$ be unit-regular.  Then all unit inner inverses of $\overline{a}$ lift to unit inner inverses of $a$.
\end{lemma}

\begin{proof}
Let $\overline{w}$ be a unit inner inverse for $\overline{a}$.  We may as well assume $w\in R$ is a lift of $\overline{w}$ that is also an inner inverse for $a$ because inner inverses lift. Set $g:=aw$ and $h:=wa$. Let $w_1 = hwg$ and $w_2:=(1-h)w(1-g)$. Then $w = w_1 + w_2$ because the cross terms $hw(1-g)$ and $(1-h)wg$ in the Peirce decomposition relative to $h, 1-h$ and $g,1-g$ are both zero.  Since $a$ is unit-regular, we know that $(1-g)R\cong (1-h)R$.  Fix elements $s\in (1-h)R(1-g),\,  t\in (1-g)R(1-h)$ such that left multiplication by these elements are inverse isomorphisms between $(1-g)R$ and $(1-h)R$.  Now, $\overline{tw_2}\in U(\overline{(1-g)R(1-g)})$.  By hypothesis, $\overline{tw_2}$ lifts to a unit $v$ of $(1-g)R(1-g)$.

Consider $w_1+sv$.  First, it is an inner inverse for $a$, since
\[
a(w_1+sv)a=a(waw+sv(1-g))a=awawa+0=a.
\]
Second, it is a unit; its inverse is simply $a+v^{-1}t$.  Finally, $w_1+sv$ is a lift of $\overline{w}$ since
\[
\overline{w_1+sv}=\overline{w_1+stw_2}=\overline{w_1+w_2}=\overline{w}.
\]
\end{proof}

\begin{theorem}\label{T:surjective}
Given any inverse system $\varprojlim R_i$ in the variety {\bf Ring}, where each homomorphism is surjective, and where each ring is a regular ring, then the inverse limit is a regular ring.  The same is true for  unit-regular rings and for exchange rings.
\end{theorem}

\begin{proof}
Fix any inner inverse operation, or unit inner inverse operation, on $R_1$.  Applying Lemma \ref{L:Lifting} we can recursively find similar operations on $R_2,R_3,\ldots$ so that the connecting homomorphisms respect these operations. The inverse limits are now taking place in the varieties \textbf{Reg} and \textbf{UnitReg} respectively, whence that is also where the limit must reside. The result for exchange rings was obtained by Pedersen and Perera \cite[Theorem 1.4]{PP}.
\end{proof}
\prskip

\begin{question}
Does  Theorem \ref{T:surjective} also hold for separative regular (or separative exchange) rings? Here is a partial answer, which we extend to exchange rings later (Theorem \ref{T:weak lifting exchange}).
\end{question}
\thmskip

\begin{theorem}\label{T:weak lifting}
Let $(R_i,f_i)_{i \in \mathbb{Z}_{>0}}$ be an inverse system in the variety \emph{\textbf{Ring}}, where each $f_i$ is surjective, and where each $R_i$ is a separative regular ring.  Assume that for any idempotent $e$ in any $R_{i+1}$, all units of $f_i(e) R_i f_i(e)$ lift (along $f_i$) to units of $e R_{i+1} e$.  Then the inverse limit is a separative regular ring.
\end{theorem}
\pfskip

\begin{proof}  We will use the separativity criterion of [2, Proposition 6.2]: A regular ring $R$ is separative if and only if each $a \in R$ satisfying
\begin{equation}\tag{*}
 R r(a) = \ell(a) R = R(1-a)R
\end{equation}
is unit-regular in $R$. (Here $r(a)$ and $l(a)$ are the right and left annihilator ideals of the element $a$.)
Since $R r(a)$ and $\ell(a) R$ are always contained in $R(1-a)R$, the above condition is equivalent to
\begin{equation}\tag{**}
  1-a \in R r(a) \cap \ell(a) R.
\end{equation}
Thus, let $a$ be an element of the regular ring $R := \varprojlim R_i$ satisfying $(**)$.  Then
$$
1-a = \sum_{j=1}^m s_j x_j = \sum_{k=1}^n y_k t_k
$$
for some $s_j,x_j,y_k,t_k \in R$ such that $ax_j = y_ka = 0$ for all $j$, $k$.  These equations project to corresponding equations in each $R_i$, and so each component $a_i$ of $a$ satisfies the $(**)$ condition in $R_i$.  Consequently, each $a_i$ is unit-regular in $R_i$.
\prskip

Starting with a unit inner inverse for $a_1$ and applying Lemma 5.1 recursively, we obtain components for a unit inner inverse of $a$ in $R$. Therefore $R$ is separative.
\end{proof}
\thmskip

\begin{corollary}\label{f.d.}
\textbf{\emph{(to Theorem \ref{T:surjective})}}\ Let $F$ be a field.  If $(R_i,f_i)_{i\in \Z_{>0}}$ is an inverse system of finite-dimensional regular $F$-algebras (resp.\ exchange $F$-algebras), and the connecting maps are $F$-algebra homomorphisms, then the inverse limit is a unit-regular ring (resp.\ an exchange ring.)
\end{corollary}

\begin{proof}
Notice that f.d.\ regular algebras are unit-regular because they are semisimple.
For each integer $i\geq 1$, define
\[
S_i:=\bigcap_{j>i}f_if_{i+1}\cdots f_{j-1}(R_j).
\]
For any $j>i$, the image of $R_j$ in $R_i$ is an $F$-subalgebra, of possibly smaller dimension.  Since the dimension cannot decrease infinitely often, then $S_i$ is the image of $R_j$ in $R_i$ for any sufficiently large integer $j>i$.  Hence $S_i$ is a finite dimensional $F$-algebra of the same type.
The connecting maps $f_i$ restrict to connecting maps between the $S_i$.  The inverse limit is unchanged, but now the restricted connecting maps are surjective. Hence our corollary follows from Theorem \ref{T:surjective}.
\end{proof}
\prskip

The trick used in the proof of the previous corollary had been observed by Grothendieck in the 1960s (see \cite{Groth}). Later, others such as Pedersen and Perera in \cite{PP} observed one can quickly reduce to the case where the connecting maps are surjective by replacing the $R_i$ by the image $S_i$ of $\pi_i$, where $\pi_i: \varprojlim R_i \rightarrow R_i$ is the projection map. Let's formally record the result.
\thmskip

\begin{proposition} Let $R = \varprojlim R_i$ be an inverse limit in \emph{\textbf{Ring}} (or in any variety), with projection maps $\pi_i : R \rightarrow R_i$. Then $R$ can also be viewed as an inverse limit of rings (algebras) $S_i = \pi_i(R)$ with surjective connecting maps restricted from the original connecting maps $f_i : R_{i+1} \rightarrow R_i$. However, the modification of $R_i$ to $S_i$ can sometimes alter properties, such as going from unit-regular to non-unit-regular (as must happen in Construction \ref{Constr:not unit regular} by Theorem \ref{T:surjective}). In fact, the passage from $R_i$ to $S_i$ doesn't particularly respect properties that don't already pass to non-surjective inverse limits.
\end{proposition}

\begin{proof} It is clear that $f_i$ induces a surjection of $S_{i+1}$ onto $S_i$, and the inverse limit from this new system agrees with the original.
\end{proof}
\prskip

\begin{remark} \label{R:modify}
 Proposition 5.6 is somewhat dual to Proposition \ref{P:Inj} for assuming the connecting maps are injective, except there the properties of the modified $R_i$ in $\varprojlim R_i$ remain largely the same as the original (e.g.\ regularity, unit-regularity, separativity). \hfill $\square$
\end{remark}
\prskip

In the situation of Corollary \ref{f.d.}, the inverse limit turns out to be a direct product of f.d.\ regular $F$-algebras, as follows from our next result. Let us say that a ring homomorphism $g : A \rightarrow B$ is an \emph{ideal-split ring epimorphism} if $g$ is surjective and $\ker g$ is an ideal direct summand of $A$, hence also a nonunital ring direct summand.  There is then a nonunital ring homomorphism $h : B \rightarrow A$ such that $gh = 1_B$, so that $g$ splits in the category \textbf{Rng} of nonunital rings. Corresponding statements hold for $F$-algebra homomorphisms.
\prskip

\begin{proposition} \label{split}
Let $R$ be an inverse limit of a system $(R_i,f_i)_{i > 0}$ where the $f_i$ are ideal-split ring (resp., $F$-algebra) epimorphisms.  Then
$$
R \cong R_1 \times \prod_{i\ge1} \ker f_i
$$
as rings (resp., $F$-algebras).
\end{proposition}

\begin{proof}
We give the proof for inverse limits of rings. The proof for inverse limits of $F$-algebras is identical. We assume with no loss of generality that
\[
R = \{(a_i)_{i\in \Z_{> 0}}\in \prod R_i\, :\, f_i(a_{i+1})=a_i \text{ for each }i\in \Z_{> 0}\},
\]
and for $i \in \Z_{> 0}$ we let $\pi_i$ denote the projection $R \rightarrow R_i$.
For $i \ge 1$, set $K_{i+1} := \ker f_i \vartriangleleft R_{i+1}$.  By assumption, there is an ideal $L_{i+1} \vartriangleleft R_{i+1}$ such that $R_{i+1} = K_{i+1} \oplus L_{i+1}$.  Thus $L_{i+1}$ is a unital ring and $f_i$ maps $L_{i+1}$ isomorphically onto $R_i$. Now set $R_{11} := R_1$ and write $R_2 = R_{21} \oplus R_{22}$ with $R_{21} := L_2$ and $R_{22} := K_2$, so that $f_1$ maps $R_{21}$ isomorphically onto $R_{11}$.  There is an ideal decomposition $L_3 = R_{31} \oplus R_{32}$ such that $f_2$ maps $R_{31}$ and $R_{32}$ isomorphically onto $R_{21}$ and $R_{22}$, respectively.  Then $R_3 = R_{31} \oplus R_{32} \oplus R_{33}$ with $R_{33} := K_3$.
\prskip

Continuing recursively, we obtain ideal decompositions $R_{i+1} = \bigoplus_{k=1}^{i+1} R_{i+1,k}$ such that $R_{i+1,i+1} = K_{i+1}$ and for $k \in [1,i]$, the map $f_i$ restricts to an isomorphism $f_{ik} : R_{i+1,k} \rightarrow R_{ik}$.  Let $p_{i+1,k}$ denote the projection $R_{i+1} \rightarrow R_{i+1,k}$ relative to the above decomposition.  We include the case $i=0$ in this notation, so that $p_{11}$ is the identity map $R_1 \rightarrow R_{11}$.  Then $p_{i+1,k} f_{i+1} = f_{i+1,k} p_{i+2,k}$ for all $k \le i+1$. For $k \ge 1$, the map $p_k := (p_{kk} \pi_k, p_{k+1,k} \pi_{k+1}, \dots)$ projects $\prod_{i\ge1} R_i$ onto $\prod_{i\ge k} R_{ik}$.  The family $(p_k)_{k\ge1}$ induces a homomorphism $p : \prod_{i\ge 1} R_i \rightarrow \prod_{k\ge1} \bigl( \prod_{i\ge k} R_{ik} \bigr)$.  If $x\in \prod_{i\ge1} R_i$, then
\begin{itemize}
\item Each $x_i = x_{i1} + \cdots + x_{ii}$ for some $x_{ik} \in R_{ik}$.
\item $p_{ik}(x_i) = x_{ik}$ for $i \ge k \ge 1$.
\item $p_k(x) = (x_{kk}, x_{k+1,k}, \dots)$ for $k \ge1$.
\end{itemize}
Hence, in case $x \in \ker p = \bigcap_{k\ge1} \ker p_k$, we have $x_{ik} = 0$ for all $i \ge k \ge 1$, and consequently $x = 0$.  Therefore $p$ is injective.
For $k \ge 1$, let $S_k \subseteq \prod_{i\ge k} R_{ik}$ denote the inverse limit of the system $(R_{ik},f_{ik})_{i\ge k}$.  Since $f_{ik} : R_{i+1,k} \rightarrow R_{ik}$ is an isomorphism for all $i \ge k$, we have $S_k \cong R_{kk}$, which equals $R_1$ when $k=1$ and $K_k$ when $k>1$.  From the fact that $p_{ik} f_i = f_{ik} p_{i+1,k}$ for all $i \ge k$, we obtain $p_k(R) \subseteq S_k$, and thus $p(R) \subseteq S := \prod_{k\ge1} S_k$.
\prskip

We claim that $p(R) = S$.  Thus let $s = (s_k)_{k\ge1} \in S$, and write $s_k = (s_{ik})_{i\ge k} \in S_k$ for $k \ge 1$.  Then $f_{ik}(s_{i+1,k}) = s_{ik}$ for all $i \ge k \ge 1$.  Set $x_i := s_{i1} + \cdots + s_{ii} \in R_{i1} + \cdots + R_{ii} = R_i$ for all $i \ge 1$, and observe that
$$
f_i(s_{i+1,1} + \cdots + s_{i+1,i}) = f_{i1}(s_{i+1,1}) + \cdots + f_{ii}(s_{i+1,i}) = s_{i1} + \cdots + s_{ii} = x_i \,.
$$
Since $x_{i+1} = s_{i+1,1} + \cdots + s_{i+1,i} + s_{i+1,i+1}$ with $s_{i+1,i+1} \in R_{i+1,i+1} = \ker f_i$, it follows that $f_i(x_{i+1}) = x_i$.  Thus $x := (x_i)_{i\ge1}$ lies in $R$.  Now
$$
p_k(x) = (x_{kk}, x_{k+1,k}, \dots) = (s_{kk}, s_{k+1,k}, \dots) = s_k \qquad \forall\; k \ge1,
$$
whence $p(x) = (s_1,s_2,\dots) = s$.  Thus $p(R) = S$, as claimed.
\prskip

Therefore $p$ restricts to an isomorphism of $R$ onto $S$.  Since
$$
S = \prod_{k\ge1} S_k \cong \prod_{k\ge1} R_{kk} = R_1 \times \prod_{k\ge1} K_{k+1} = R_1 \times \prod_{k\ge1} \ker f_k = R_1 \times \prod_{i\ge1} \ker f_i \,,
$$
the proposition is proved.
\end{proof}
\prskip

In particular, if the $R_i$ are semisimple rings, then for the $f_i$ to be ideal-split ring epimorphisms they just need to be surjective.  The proposition shows that if the $R_i$ are semisimple, then $R$ is isomorphic to a direct product of semisimple rings. In particular, the limit is unit-regular, yielding another proof of Corollary \ref{f.d.}.
\prskip

\begin{remark}\label{R:C*}
Our results on algebraic inverse limits of regular (or exchange) rings do not appear to impact their operator algebra cousins, say $C^*$-algebras. One reason is an old result of Kaplansky (see \cite{Kap}) saying that Banach algebras that are regular have to be finite-dimensional. Therefore, in general, (algebraic) inverse limits $\varprojlim S_i$ of $C^*$-algebras need not be  $C^*$-algebras (even with surjective connecting maps). As a simple example, for $i = 1,2, \ldots$, let $R_i = M_2(\mathbb{C})$ with standard involution and norm, and set $S_i = R_1 \times R_2 \times \cdots \times R_i$. Then with the natural projections $f_i: S_{i+1}   \rightarrow S_i$ as connecting maps,
\[
                         \varprojlim S_i \ \cong \ \prod _{i = 1}^{\infty}\, R_i
\]
is an infinite-dimensional, regular algebra, whence not a Banach algebra let alone a $C^*$-algebra. However, Brown and Pedersen (see \cite{BP}) have shown that an inverse limit in the analytic sense (so strings $(a_1,a_2, \ldots )$ in the inverse limit are bounded) of $C^*$-algebras of real rank 0, and with surjective connecting maps, is again a $C^*$-algebra of real rank 0. This fits neatly with algebraic inverse limits of exchange rings because $C^*$-algebras of real rank 0 are exactly the $C^*$-algebras which are exchange rings by \cite[Theorem 7.2]{AGOP1}. \hfill $\square$
\end{remark}
\thmskip


\section{Inverse limits of associated monoids}
\thmskip

Much of the work in regular rings $R$ since the 1990s has been done by applying monoid techniques to the commutative monoid $V(R)$ of isomorphism classes $[A]$ of f.g.\ projective right $R$-modules $A$, where addition is defined by $[A]+[B] = [A \oplus B]$. Alternatively, we can view $V(R)$ as the monoid of isomorphism classes $[e]$ of idempotents $e$ from $\bigcup_{n=1}^\infty M_n(R)$, where $[e]$, for an idempotent $e \in M_n(R)$, corresponds to the isomorphism class $[eR^n]$. So a natural question is how  our study of inverse limits of regular rings relates to inverse limits of these associated monoids. This is particularly relevant to the Separativity Problem because the inverse limit of separative monoids is always separative, being a submonoid of a product of separative monoids. Thus, given an inverse limit $R = \varprojlim R_i$ of separative regular rings $R_i$, if we know $V(R) \cong \varprojlim V(R_i)$, we can immediately conclude that $R$ is separative. On the other hand, if $V(R$) does not match the limit of the $V(R_i)$, there is some hope that $R$ may not be separative.
\prskip

We recall two properties of monoids.  Given a commutative monoid $M$, written additively, then it is \emph{conical} when for any $a,b\in M$,
\[
a+b=0 \Longrightarrow a=b=0.
\]
Also, an element $u\in M$ is called an \emph{order-unit} if for each $a\in M$, there exists some nonnegative integer $n\in \Z_{\geq 0}$ such that $a\leq nu$, which means
\[
a+b=nu\, \text{ for some $b\in M$}.
\]
Notice that $V(R)$ is conical, and the isomorphism class of the right regular module $R_R$ is an order-unit.
Conversely, any conical commutative monoid with an order-unit is isomorphic to $V(R)$ for some ring $R$, with the order-unit mapping to $[R_R]$ (and the ring can be forced to satisfy extra properties); this deep result is due to the work of Bergman and Dicks; see the paragraph following Theorem 3.4 in \cite{BergmanDicks}. \prskip

We can view $V(R)$ as a functorial construction.  Indeed, given any ring homomorphism $\varphi\colon R\to S$, there is a corresponding ring homomorphism $M_n(\varphi)\colon M_n(R)\to M_n(S)$, for each integer $n\geq 1$.  Thus, we can define $V(\varphi)\colon V(R)\to V(S)$ by the rule $[e]\mapsto [M_n(\varphi)(e)]$ for any idempotent $e \in M_n(R)$.
\prskip

Let $\Monu$ denote the category whose objects are pairs $(M,u)$, where $M$ is a commutative monoid and $u$ is an order-unit in $M$.  The morphisms are monoid homomorphisms that respect the distinguished order-units. This category is not a variety by Birkhoff's theorem (see \cite[Theorem 2.15]{BasicAlgebra}), since the class of objects is not closed under infinite direct products.  Thus, we need to provide an alternative argument for why inverse limits exist in $\Monu$.
\prskip

Suppose that for each $i\in \Z_{>0}$ we are given a pair $(M_i,u_i)$, as well as connecting morphisms $\varphi_i\colon (M_{i+1},u_{i+1})\to (M_i,u_i)$ in $\Monu$.  The Cartesian product $\prod_{i\in \Z_{>0}}M_i$ is a monoid containing the element $u=(u_i)_{i\in \Z_{>0}}$; however, $u$ may not be an order-unit in the product.  Let $M$ be the collection of elements $a$ in the product such that $a\leq nu$ for some $n\in \Z_{>0}$.  Then $M$ is a monoid with $u$ as an order-unit; it is the product object of the family $((M_i,u_i))_{i\in \Z_{>0}}$ in $\Monu$.
Next, let $I$ be the inverse limit object of the inverse system $((M_i,\varphi_i))$, taken in the variety of monoids.  In other words,
\[
I=\biggl\{(a_i)_{i\in \Z_{>0}}\in \prod_{i\in \Z_{>0}}M_i\, :\, \varphi_i(a_{i+1})=a_i \text{ for each $i\in \Z_{>0}$}\biggr\}.
\]
Now, fix $N := I\cap M$, which is a submonoid of $M$ containing $u$.  We claim that $u$ is an order-unit in $N$.  Given $a \in N$, we have $a = (a_i)$ with $\varphi_i(a_{i+1})=a_i$ for all $i$, and $a+b = mu$ for some $b \in \prod_{i\in \Z_{>0}} M_i$ and $m \in \Z_{>0}$.  For each $i$, we have
\[
a_i + b_i = m u_i = \varphi_i(m u_{i+1}) = \varphi_i(a_{i+1} + b_{i+1}) = a_i + \varphi_i(b_{i+1}).
\]
Adding $b_i$ to each end of these equations yields
\[
m u_i + b_i = m u_i + \varphi_i(b_{i+1}) = \varphi_i(m u_{i+1} + b_{i+1}).
\]
Thus, $c := (m u_i + b_i)_{i\in \Z_{>0}}$ is an element of $I$ with $a+c = 2m u$, so that $c \in N$ and $a \le_N 2m u$.  This shows that $u$ is an order-unit in $N$.
 It is straightforward to check that $(N,u)$, together with the projection maps, is an inverse limit for the system $((M_i,u_i),\varphi_i)$ in $\Monu$.

 One can also start with $I$, take $I_u$ to be the set of those $a \in I$ such that $a \le_I nu$ for some $n \in \Z_{>0}$, and then show that $(I_u,u)$ with the projection maps is an inverse limit in $\Monu$ for the given inverse system.  The argument above amounts to showing that $I_u = I \cap M$.
\prskip

Our first result is that the $V$-functor is respected in surjective inverse limits of unit-regular rings.  Before we prove that, we need the following result of independent interest, which complements much of the material in \cite[Example 3.7]{M}.  Given a ring $R$ with idempotents $e,f\in R$, write $e\cong f$ when they are isomorphic, and write $e\sim f$ when they are conjugate. Conjugate idempotents are always isomorphic, and the fact that the converse holds in unit-regular rings is well-known.

\begin{lemma}[{cf.\ \cite[Lemmas 5 and 9]{GoodearlUnpublished}}]\label{Lemma:LiftIsosUnitReg}
Let $R$ be a ring, let $I\trianglelefteq R$ be an ideal, and let $\pi\colon R\to \overline{R}:=R/I$ be the natural quotient map.  Let $e,f\in R$ be idempotents.
\begin{itemize}
\item[\textup{(1)}] If $\,\overline{e}\sim p$ for some idempotent $p\in \overline{R}$, and if units lift from $\overline{R}$ to $R$, then we can choose an idempotent $g\in R$ satisfying $\overline{g}=p$ and $e\sim g$.
\item[\textup{(2)}] Assume $e\cong f$.  Then every isomorphism $\overline{e}\overline{R}\to \overline{f}\overline{R}$ lifts to an isomorphism $eR\to fR$ if and only if all units lift from $\overline{e}\overline{R}\overline{e}$ to $eRe$.
\end{itemize}
If $R$ is a unit-regular ring, then both lifting hypotheses hold.
\end{lemma}
\begin{proof}
(1) Fix $v\in U(\overline{R})$ with $p=v^{-1}\overline{e}v$.  Let $u\in U(R)$ be a lift of $v$.  Take $g:=u^{-1}eu$.

(2) Since $e\cong f$,  we can fix $a\in eRf$ and $b\in fRe$ such that $ab=e$ and $ba=f$.  Left multiplications by $\overline{a}$ and $\overline{b}$ induce inverse isomorphisms between $\overline{e}\overline{R}$ and $\overline{f}\overline{R}$.

$(\Longrightarrow)$: If $v$ is a unit of $\overline{e}\overline{R}\overline{e}$, then (multiplication by) $\overline{b} v$ and $v^{-1} \overline{a}$ give inverse isomorphisms between $\overline{e}\overline{R}$ and $\overline{f}\overline{R}$.  By assumption, $\overline{b} v$ lifts to some $w \in fRe$ which provides an isomorphism $eR \to fR$, say with inverse provided by $w' \in fRe$.  Then $aw$ is a unit of $eRe$ (with inverse $w'b$), and $\overline{aw} = \overline{a} \overline{b} v = v$.

$(\Longleftarrow)$: Given an isomorphism $\varphi\colon \overline{e}\overline{R}\to \overline{f}\overline{R}$, we can view it as left multiplication by an element $x\in \overline{f}\overline{R}\overline{e}$, whose inverse map is left multiplication by some $y\in \overline{e}\overline{R}\overline{f}$.  In particular, $xy=\overline{f}$ and $yx=\overline{e}$.

Now, we find
\begin{eqnarray*}
(y\overline{b})(\overline{a}x)=y\overline{f}x=yx=\overline{e} \quad\text{ and }\quad
(\overline{a}x)(y\overline{b})=\overline{a}\overline{f}\overline{b}=\overline{ab}=\overline{e}.
\end{eqnarray*}
So, $y\overline{b}\in U(\overline{e}\overline{R}\overline{e})$, with inverse $\overline{a}x$.  By assumption, $y\overline{b}$ lifts to some $w\in U(eRe)$, say with inverse $w'$.  Then setting $u:=wa\in eRf$ and $v:=bw'\in fRe$, they satisfy
\begin{eqnarray*}
& & uv = wew'=ww'=e,\\
& & vu=bw'wa=bea=ba=f, \text{ and}\\
& & \overline{v}=\overline{bw'}=\overline{b} (y \overline{b})^{-1} = \overline{b} \overline{a} x= x.
\end{eqnarray*}
So, left multiplication by $v$ is an isomorphism $eR\to fR$ that lifts $\varphi$.

Finally, we prove the last sentence.  As unit-regularity passes to corner rings, the lifting hypotheses follow from \cite[Lemma 3.5]{Bac}.
\end{proof}

\begin{theorem}[{\cite[Proposition 7]{GoodearlUnpublished}}]\label{Thm:VInvLimit}
Let $(R_i,\varphi_i)$ be an inverse system of unit-regular rings, where each $\varphi_i$ is a surjective ring homomorphism.  Fix $R=\varprojlim R_i$, with projection maps $\pi_i\colon R\to R_i$.  Also fix $(N,u)$ to be an inverse limit in $\Monu$ of the corresponding inverse system $((V(R_i),[R_i]),V(\varphi_i))$, with projection maps $p_i\colon N\to V(R_i)$.

There is a unique $\Monu$-morphism $\eta\colon (V(R),[R])\to (N,u)$ such that $p_i\eta=V(\pi_i)$ for each $i\in \Z_{>0}$.  Moreover, $\eta$ is a $\Monu$-isomorphism.
\end{theorem}
\begin{proof}
The morphisms
\[
V(\pi_i)\colon (V(R),[R])\to (V(R_i),[R_i])
\]
satisfy $V(\varphi_i)V(\pi_{i+1})=V(\pi_{i})$ for each integer $i\geq 1$, by functoriality of $V$, together with the equalities $\varphi_i\circ \pi_{i+1}=\pi_i$.  Thus, the existence and uniqueness of $\eta$ are clear, from the universal property of inverse limits.

Next we prove injectivity.  Consider $v,w\in V(R)$ such that $\eta(v)=\eta(w)$.  Then there exist idempotents $e,f\in M_n(R)$ for some integer $n\geq 1$ such that $v=[e]$ and $w=[f]$.  Write $e=(e_i)$ and $f=(f_i)$ for idempotents $e_i,f_i\in M_n(R_i)$, by identifying $M_n(R)$ with the inverse limit of the system $(M_n(R_i),M_n(\varphi_i))$.  Then
\[
[e_i]=V(\pi_i)([e])=p_i\eta(v)=p_i\eta(w)=V(\pi_i)([f])=[f_i]
\]
for each integer $i\geq 1$.  Thus, $e_i$ and $f_i$ are isomorphic in $M_n(R_i)$ for each integer $i\geq 1$.

Fix an isomorphism $e_1M_n(R_1)\to f_1M_n(R_1)$, viewed as left multiplication by an element $x_1\in f_1M_n(R_1)e_1$, say with inverse $y_1\in e_1M_n(R_1)f_1$.  By using Lemma \ref{Lemma:LiftIsosUnitReg}(2) and the fact that matrix rings over unit-regular rings are still unit-regular, then there is an element $x_2\in f_2M_n(R_2)e_2$ such that left multiplication by $x_2$ yields an isomorphism
\[
e_2M_n(R_2)\to f_2M_n(R_2),
\]
and $\varphi_1(x_2)=x_1$.  The inverse map is left multiplication by some $y_2\in e_2M_n(R_2)f_2$ that lifts $y_1$.  Recursively repeating this process, we can create a compatible sequence of elements $x=(x_i)\in M_n(R)$, whose inverses form a compatible sequence $y=(y_i)\in M_n(R)$.  Since $e=xy$ and $f=yx$, then $e\cong f$, hence $v=[e]=[f]=w$ as desired.

Finally, we prove surjectivity.  Fix some arbitrary $t=(t_i)\in N$.  Then $t\leq nu$ for some integer $n\geq 1$, and so (for each integer $i\geq 1$) we have $t_i\leq n[R_i]$.  Hence, $t_i=[e_i]$ for some idempotent $e_i\in M_n(R_i)$.  Since $V(\varphi_i)(t_{i+1})=t_i$, we have $\varphi_i(e_{i+1})\cong e_i$.

Fix $f_1=e_1$.  By part (1) and the last sentence of Lemma \ref{Lemma:LiftIsosUnitReg}, there is an idempotent $f_2\in M_n(R_2)$ with $e_2\cong f_2$ and $\varphi_1(f_2)=f_1$.  Recursively repeating this process, there is a compatible sequence of idempotents $f=(f_i)\in M_n(R)$ with $f_i\cong e_i$ for each integer $i\geq 1$.  Since
\[
\pi_i\eta([f])=V(\pi_i)([f])=[f_i]=[e_i]=\pi_i(t)
\]
for each integer $i\geq 1$, we see that $\eta([f])=t$.  Therefore, $\eta$ is surjective.
\end{proof}
\prskip

\begin{remark}
The surjectivity hypothesis in Theorem \ref{Thm:VInvLimit} is certainly not superfluous.  Indeed, Construction 4.3 produced an inverse limit of unit-regular rings that is regular but not unit-regular.  However, a regular ring $R$ is unit-regular if and only if the monoid $V(R)$ is cancellative. So $V(R)$ cannot be isomorphic to the inverse limit of the $V(R_i)$, otherwise the limit is a submonoid of the cancellative monoid $\prod V(R_i)$, whence cancellative. In Section 7, we show that even when the connecting maps are surjective, the monoid of the inverse limit may not match the inverse limit of the individual $V(R_i)$. This suggests there are limitations to what inverse limits of the associated monoids $V(R_i)$ can tell us about $\varprojlim R_i$. \hfill $\square$
\end{remark}
\prskip

Theorem \ref{Thm:VInvLimit} and its proof naturally generalize to give:

\begin{corollary}
Let $(R_i,\varphi_i)$ be an inverse system of unit-regular rings, where each $\varphi_i$ is a surjective ring homomorphism.  Fix $R=\varprojlim R_i$, with projection maps $\pi_i\colon R\to R_i$.  Also fix $(G,u)$ to be an inverse limit, in the category of partially ordered abelian groups with distinguished order-unit, of the corresponding inverse system $((K_0(R_i),[R_i]),K_0(\varphi_i))$, with projection maps $p_i\colon G\to K_0(R_i)$.

There is a unique morphism $\eta\colon (K_0(R),[R])\to (G,u)$ such that $p_i\eta=K_0(\pi_i)$ for each $i\in \Z_{>0}$.  Moreover, $\eta$ is an isomorphism of partially ordered abelian groups.
\end{corollary}

A careful examination of the proof of injectivity in Theorem \ref{Thm:VInvLimit} shows  we can weaken the hypothesis that the $R_i$ are unit-regular; we only need the ability to lift units through corners of matrix rings.  We record that result, as it will have some bearing on separativity for rings.

\begin{proposition}\label{Prop:MatrixUnitCornerLifting}
Let $(R_i,\varphi_i)_{i \in \Z_{>0}}$ be an inverse system of rings, where each $\varphi_i$ is a surjective ring homomorphism.  Fix $R=\varprojlim R_i$, with projection maps $\pi_i\colon R\to R_i$.  Assume that for any $i,n\in \Z_{>0}$, and for any idempotent $e\in M_n(R_{i+1})$, all units of $\varphi_i(e)M_n(R_i)\varphi_i(e)$ lift to units of $eM_n(R_{i+1})e$.  Also fix $(N,u)$ to be an inverse limit in $\Monu$ of the inverse system $((V(R_i),[R_i]),V(\varphi_i))_{i \in \Z_{>0}}$, with projection maps $p_i\colon N\to V(R_i)$.

The unique morphism $\eta\colon (V(R),[R])\to (N,u)$ such that $p_i\eta = V(\pi_i)$ for each $i \in \mathbb{Z}_{>0}$ is injective.
\end{proposition}

\begin{corollary}\label{C:SepSurjective}
Using the same notation from \textup{Proposition \ref{Prop:MatrixUnitCornerLifting}}, and assuming the same lifting conditions, if each $R_i$ is separative, then $R$ is separative.
\end{corollary}
\begin{proof}
Given a ring $R$, then by definition, to say that $R$ is separative means that $V(R)$ satisfies
\[
[e]+[e]=[e]+[f]=[f]+[f]\Longrightarrow [e]=[f]
\]
for any two idempotents $e,f$ in matrix rings over $R$.

Assume the premise of the implication, in the case when $R=\varprojlim R_i$.  Write $\eta([e])=([e_i])$ and $\eta([f])=([f_i])$ as compatible sequences from the monoids $V(R_i)$.  Then (for each integer $i\geq 1$) note that $[e_i]+[e_i]=[e_i]+[f_i]=[f_i]+[f_i]$.  Separativity implies that $[e_i]= [f_i]$ for each $i\geq 1$.  Hence $\eta([e])=\eta([f])$. Injectivity of $\eta$ yields $[e]=[f]$.
\end{proof}
\prskip

For inverse systems of regular rings, Theorem \ref{T:weak lifting} establishes Corollary \ref{C:SepSurjective} with weaker lifting conditions.  The lifting conditions can also be weakened for separative exchange rings, as follows, so that Theorem \ref{T:weak lifting} can be extended to inverse systems of exchange rings (Theorem \ref{T:weak lifting exchange}).
\prskip

Recall that an ideal $J$ of a ring $R$ is called a \emph{trace ideal} if there exists a finitely generated projective $R$-module $P$ such that $J = \sum \{ f(P) : f \in \Hom_R(P,R) \}$, the \emph{trace ideal of $P$}. If $R$ is an exchange ring, then the trace ideals are exactly the ideals generated by a single idempotent. Indeed, since $R$ is exchange, we have
$P\cong e_1R\oplus \cdots \oplus e_nR$ for some idempotents $e_1,\dots ,e_n\in R$, and then the trace ideal of $P$ is $J = Re_1R+\cdots +Re_nR$. By \cite[Lemma 2.1]{AGOR}, there exists an idempotent $e\in R$ such that $J= ReR$.  Conversely, $ReR$ is the trace ideal of $eR$, for any idempotent $e \in R$.

\begin{proposition}  \label{P:LiftSepExch}
 Let $R$ be a separative exchange ring and let $I$ be an ideal of $R$. Set $\overline{R}:= R/I$ and denote by $\overline{x}$, for $x\in R$, the class of $x$ in $\overline{R}$.  Let $E$ be a set of idempotents in $R$ such that the collection $\{ ReR : e \in E \}$ contains all trace ideals of $R$.
 Then the following conditions are equivalent.
 \begin{enumerate}
  \item For any $n>0$ and for any idempotent $f\in M_n(R)$, all units of $\overline{f}M_n(\overline{R})\overline{f}$ lift to units of $fM_n(R)f$.
  \item For any idempotent $e \in E$, all units of $\overline{e}\overline{R} \overline{e}$ lift to units of $eRe$.
  \item For each $e \in E$, the natural map $K_1(eRe) \to K_1(\overline{e}\overline{R}\overline{e})$ is surjective.
 \end{enumerate}
 \end{proposition}

  \begin{proof}
  (1)$\Longrightarrow$(2) is obvious.

  (2)$\Longrightarrow$(3): Let $e \in E$, and assume that all units of $\overline{e}\overline{R} \overline{e}$ lift to units in $eRe$.
  Since the natural map $GL_1(\overline{e}\overline{R}\overline{e}) \to K_1(\overline{e}\overline{R}\overline{e})$ is surjective by \cite[Theorem 2.8]{AGOR}, it follows that the map $K_1(eRe)\to K_1(\overline{e}\overline{R}\overline{e})$ is surjective.

  (3)$\Longrightarrow$(1): Let $n>0$ and let $f$ be an idempotent in $M_n(R)$. Let $J$ be the ideal of $R$ generated by the entries of $f$. Then $J$ is a trace ideal, corresponding to the finitely generated projective module $P:=fR^n$, so there is an idempotent $e\in E$ such that $J=ReR$.  By (3), the natural map $K_1(eRe) \to K_1(\overline{e}\overline{R}\overline{e})$ is surjective. Observe that $eRe$ and $fM_n(R)f$ are Morita-equivalent unital rings, and so are $\overline{e}\overline{R}\overline{e}$ and $\overline{f}M_n(\overline{R}) \overline{f}$. There are isomorphisms $K_1(eRe)\to K_1(fM_n(R)f)$ and $K_1(\overline{e}\overline{R}\overline{e}) \to K_1(\overline{f}M_n(\overline{R}) \overline{f})$ such that the following diagram is commutative:
\[
\xymatrix{
K_1(eRe) \ar[r] \ar[d]_{\cong} &K_1(\overline{e}\overline{R}\overline{e}) \ar[d]^{\cong}  \\
K_1(fM_n(R)f) \ar[r] &K_1(\overline{f}M_n(\overline{R})\overline{f}).
}
     \]
 It follows that the map $K_1(fM_n(R)f)\to K_1(\overline{f}M_n(\overline{R})\overline{f})$ is surjective. Hence the connecting map
  $\delta \colon K_1(\overline{f}M_n(\overline{R})\overline{f}) \to K_0(fM_n(I)f)$ is zero, and it follows from \cite[Theorem 2.4]{Perera} that all units in $\overline{f}M_n(\overline{R})\overline{f}$ lift to units in $fM_n(R)f$.
 \end{proof}

\begin{theorem}\label{T:weak lifting exchange}
Let $(R_i,\varphi_i)_{i \in \mathbb{Z}_{>0}}$ be an inverse system in the variety \emph{\textbf{Ring}}, where each $\varphi_i$ is surjective, and where each $R_i$ is a separative exchange ring.  For each $i$, let $E_i$ be a set of idempotents in $R_i$ such that the collection $\{ R_i e R_i : e \in E_i \}$ contains all trace ideals of $R_i$, and assume that for any idempotent $e \in E_{i+1}$, all units of $\varphi_i(e) R_i \varphi_i(e)$ lift to units of $e R_{i+1} e$.  Then the inverse limit is a separative exchange ring.
\end{theorem}

\begin{proof}
Proposition \ref{P:LiftSepExch} and Corollary \ref{C:SepSurjective}.
\end{proof}
\thmskip

The upcoming paper \cite{AGNOPP} presents a comprehensive account of stable rank and cancellation properties of monoids under various hypotheses, along with applications to various classes of modules.
\thmskip

\section{An instructive example}
\thmskip

In this section, we aim to construct an example of a surjective inverse limit of regular rings $Q_i$ whose associated monoid does not match the inverse limit of the monoids of the $Q_i$, that is, $V(\varprojlim Q_i) \not\cong \varprojlim V(Q_i)$. This contrasts sharply with the case where the $Q_i$ are unit-regular (Theorem \ref{Thm:VInvLimit}).
 \prskip

We first recall the definition of a separated graph, following \cite{AG12}. We will use the notation in \cite{AG12} concerning directed graphs. See, for example, \cite{AAM} for general graph algebra results.  In particular, $s,r\colon E^1\to E^0$ will denote the source and range maps, respectively, of a graph $E=(E^0,E^1)$ with vertex set $E^0$ and edge set $E^1$.

 \begin{definition}[\cite{AG12}] \label{defsepgraph}
 	A \emph{separated graph} is a pair $(E,C)$ where $E$ is a directed graph,  $C=\bigsqcup
 	_{v\in E^ 0} C_v$, and
 	$C_v$ is a partition of $s^{-1}(v)$ (into pairwise disjoint nonempty
 	subsets) for every vertex $v$. (In case $v$ is a sink, we (necessarily) take $C_v= \emptyset $.)
 	
 	If all the sets in $C$ are finite, we say that $(E,C)$ is a \emph{finitely separated} graph.
	
	Corresponding to $(E,C)$ is a commutative \emph{graph monoid} $M(E,C)$ \cite[Definition 4.1]{AG12}.  When $(E,C)$ is finitely separated, $M(E,C)$ can be presented with $E^0$ as a set of generators and relations $v = \sum\{ r(e) : e \in s^{-1}(v) \}$ for $v \in E^0$ \cite[p.186]{AG12}.
 	 \end{definition}

   We now introduce a version of the construction used in \cite{ABP} (see also \cite{ABPS}). We fix a field $K$. We consider a particular class of finitely separated graphs, as follows.

  Let $(E,C)$ be a separated graph such that
  $$E^0 := \{ v\}\sqcup W,$$
  where $W$ is a finite or countable set. Moreover, let
  $$E^1:= \{ \alpha _i : 1\le i < N\} \sqcup \{ \beta_i^j : 1\le i< N, 1\le j \le g(i)\}$$
where $2\le N \le \infty$ and $1\le g(i) <\infty$ for all $i$. Moreover we set
$s(\alpha_i)= v= r(\alpha_i)$, $s(\beta_i^j)= v$ for all $i,j$, and
$W= \bigcup_{1\le i < N} \{r(\beta_i^j) : 1\le j \le g(i) \}$.

Now the set $C$ is defined as $C:=\{X_i : 1\le i < N\}$, where
$$X_i := \{\alpha _i,\beta_i^1,\dots , \beta_i^{g(i)}\}.$$

We consider a set of commuting indeterminates $\{t^w_i: i=1,2,3 \dots \}$ at each vertex $w\in W$, such that $wt^w_i = t^w_i = t_w^i w$, and the quotient field $L_w :=K(t^w_i)$ of the polynomial ring $K[t^w_i]$. The algebra $L_w$ has unit $w$ for all $w\in W$, and has a unique $K$-algebra involution $*$ such that $(t_i^w)^* = (t_i^w)^{-1}$ for all $i$.

Let $R$ be the $*$-algebra over $K$ generated by $E^0$, $E^1$ and $L_w$ for all $w\in W$, subject to the following relations:

\begin{enumerate}
	\item $uu'= \delta_{u,u'}u$ and $u^*= u$ for all $u,u'\in E^0$,
	\item $s(e)e= e= er(e)$ for all $e\in E^1$,
	\item $\alpha_i ^* \alpha _i= v$ for all $i$,
	\item $\alpha _i \alpha _j = \alpha_j \alpha_i$ and  $\alpha _i \alpha _j^* = \alpha_j^* \alpha_i$ for all $i\ne j$,
	\item $v= \alpha_i \alpha _i^* + \sum_{j=1}^{g(i)} \beta_i^j(\beta_i^j)^*$ for all $i$,
	 \item $(\beta_i^j)^* \beta_s^t = 0$ if $(i,j)\ne (s,t)$ and $(\beta_i^j)^* \beta_i^j = r(\beta_i^j)$ for all $i,j$,
	 \item $\alpha_i^*\beta_i^s= 0 = (\beta_i^s)^* \alpha_i$ for all $i,s$,
	 \item For all $i\ne j$, $\alpha_i\beta_j^s  = \beta_j^s t_i^{r(\beta_j^s)}$ and $\alpha_i^* \beta_j^s = \beta_j^s (t_i^{r(\beta_j^s)})^{-1}$.
	  \end{enumerate}
Note that these are relations as a $*$-algebra, so all the relations obtained by applying the involution in the above equations hold in $R$.

For each finite subset $F$ of $\{ i : 1 \le i < N \}$, we consider the separated graph $(E_F,C_F)$ obtained by considering only the edges in $\bigcup_{i \in F} X_i$ and the corresponding source and range vertices, and where $C_F := \{ X_i : i \in F \}$.
Let $u_F= \sum_{u\in E_F^0} u$. Then $u_F$ is the unit for the corresponding algebra $R_F$. Let $\Sigma_F$ be the set of all polynomials $f$ in $K[x_i : i \in F]$ such that $x_i$ does not divide $f$ for any $i\in F$. Since the $\alpha_i$, $i\in F$, are commuting elements of $vR_Fv$, there is an evaluation map $f\mapsto f(\alpha_i)$, and we may consider the universal localization
$$Q_K(E_F,C_F) := \Sigma_F^{-1}R_F,$$
which is a unital $K$-algebra.  If $F\subset F'$, there is a natural (not necessarily unital) algebra homomorphism $Q_K(E_F,C_F)\to Q_K(E_{F'}, C_{F'})$, and we can define the direct limit $Q_K(E,C) :=  \varinjlim Q_K(E_F,C_F)$. We set $\Sigma := \bigcup_F \Sigma_F$.

The $K$-algebra $Q_K(E,C)$ is unital if and only if $W$ is finite.

\begin{theorem}\label{T:Qreg}
	The $K$-algebra $Q_K(E,C)$ defined above is a strongly separative regular ring and moreover
	the natural map $M(E,C)\to  V (Q_K(E,C))$ is a monoid isomorphism. \rm{(A ring is \textbf{strongly separative} if $A \oplus A \cong A \oplus B  \Longrightarrow  A \cong B$ for f.g.\ projective modules $A,B$.)}
	\end{theorem}

\begin{proof}
	For finite $F$, the $K$-algebra $Q_F := Q_K(E_F,C_F)$ is regular and the natural map $M(E_F,C_F) \to V (Q_F)$ is an isomorphism by the main result in \cite{ABP}. Passing to direct limits, we obtain regularity of $Q_K(E,C)$ and the stated monoid isomorphism.
	
	To see that $Q_K(E,C)$ is strongly separative, it suffices to show that each $Q_F$ is strongly separative.  Take $I_F$ to be the ideal of $Q_F$ generated by $E^0_F \setminus \{v\}$, and observe that $Q_F/I_F$ is isomorphic to the rational function field $K(x_i : i \in F)$.  Also $wQ_Fw \cong L_w$, a field, for each $w \in E^0_F \setminus \{v\}$.  Thus $I_F$ is semisimple.  Inasmuch as both $I_F$ and $Q_F/I_F$ are now strongly separative regular rings, by \cite[Theorem 5.5 and Proposition 1.4]{AGOP1} we must have $Q_F$ strongly separative.
\end{proof}

Observe that for finite $F\subset F'$, the map $Q_F\to Q_{F'}$ is injective. This is due to the fact that there is no vertex in $E_F$ which is sent to $0$, and the (classes of the) vertices generate the $V$-monoid of the regular ring $Q_F$. Accordingly, we will view $Q_F$ as a subalgebra of $Q_K(E,C)$ for all finite $F\subseteq \{ i : 1 \le i < N \}$.

By \cite[Theorem 2.12]{ABP}, the elements of $Q:=Q_K(E,C)$ are $K$-linear combinations of terms of the form $\gamma  \mathfrak m \nu^*$, where
$\gamma $ is a fractional $c$-path, $\mathfrak m $ is a fractional monomial and $\nu$ is a $c$-path. Here a fractional monomial at $v$ is an expression of the form
$$f^{-1} \prod_{i=1}^n \alpha_i^{k_i}(\alpha_i^*)^{l_i},$$
where $f\in \Sigma$, $k_i, l_i\ge 0$, and a fractional monomial at a vertex $w\in W$ is just a term of the form $f(t_i^w)\in L_w$. Fractional $c$-paths are either trivial (i.e.\ vertices) or paths of the form
$$f^{-1} \alpha_i^m \beta_i^j$$
for some $f\in \Sigma$, some $i$, $j$ and some $m\ge 0$. Moreover $c$-paths are of the same form as fractional $c$-paths, but with $f=1$. See \cite[Section 2]{ABP}.

We will need also some additional formulas, see \cite[Lemma 2.8]{ABP} and \cite[Lemma 2.9]{A2010} for the proof of the next two equations: For $f\in \Sigma$ we have
\begin{equation}
	\label{eq:first}
	(v-\alpha_i\alpha_i^*) f^{-1} = (f_0')^{-1}\rho^* (v-\alpha_i\alpha_i^*) = (v-\alpha_i\alpha_i ^*) (f_0')^{-1}\rho^*,
	\end{equation}

\begin{equation}
	\label{eq:second}
	\alpha_i^* f^{-1} = f^{-1}\alpha_i^* + f^{-1}(f_0')^{-1} g \rho^* (v-\alpha_i\alpha_i^*)
   	\end{equation}
where $\rho$ is a monomial in $\{ \alpha_j : j\ne i \}$, $f_0'\in K[\alpha_j: j\ne i]\cap \Sigma $, and $g\in K[\alpha_j]$.

Using these formulas we obtain, for $f\in \Sigma$

\begin{equation}
	\label{eq:third}
	(\beta_i^s)^* f^{-1} \in L_{r(\beta_i^s)} (\beta_i^s)^*
\end{equation}

\begin{equation}
\label{eq:fourth}
(v-\alpha_i\alpha_i^*)f^{-1} \in \sum _{s=1}^{g(i)} \beta_i^sL_{r(\beta_i^s)} (\beta_i^s)^*
\end{equation}

Relation \eqref{eq:third} is shown in the proof of \cite[Theorem 2.12]{ABP}. We now show relation \eqref{eq:fourth}:
$$(v-\alpha_i \alpha_i ^*)f^{-1} = \sum_{s=1}^{g(i)} \beta_i^s(\beta_i^s)^* f^{-1} \in \sum_{s= 1}^{g(i)} \beta_i^s L_{r(\beta_i^s)} (\beta_i^s)^*,$$
where we have used condition (5) above for the first equality and \eqref{eq:third} for the second.

  \begin{example}\label{E:Pere}
   We consider a family of separated graphs $(E_n,C^n)$, for $n\ge 0$ defined as follows. For $n=0$ we set
\begin{gather*}
E_0^0 :=\{v,w_0\}, \quad E_0^1 :=\{\alpha_i,\beta_i^1: i \in \Zpos \} ,  \\
r(\alpha_i)=s(\alpha_i) := v, \,\, s(\beta_i^1) := v, \,\, r(\beta_i^1) := w_0 \text{ for all } i \in \Zpos,  \\
C^0 := \{\{\alpha_i,\beta_i^1\} :  i \in \Zpos \}.
\end{gather*}
   For $n\ge 1$, we set
\begin{gather*}
E_n^0 := \{v,w_0,w_1,\dots ,w_n\}, \quad E_n^1 :=\{\alpha_i,\beta_i^1, \beta_j^2: i \in \Zpos,\ 1\le j\le n \} ,  \\
r(\alpha_i)=s(\alpha_i) := v, \quad s(\beta_i^1)= s(\beta_j^2) := v, \quad r(\beta_i^1) := w_0 ,  \\
 r(\beta_j^2) := w_j  \text{ for all } i \in \Zpos \text{ and all } 1\le j\le n,  \\
C^n := \{\{\alpha_j,\beta_j^1,\beta_j^2\}, \{\alpha_i,\beta_i^1\} :  1\le j \le n,\ i>n \}.
\end{gather*}
  Observe that $Q_K(E_n,C^n)$ is a unital $K$-algebra, with unit $v+\sum_{i=0}^n w_i$.
    We consider maps $\pi_n \colon Q_K(E_{n+1},C^{n+1})\to Q_K(E_n,C^n)$ sending $v$ to $v$, $w_i$ to $w_i$ for $0\le i\le n$ and $w_{n+1}$ to $0$. Similarly all the edges of $E_{n+1}$ are mapped to the corresponding edges in $E_n$, except for $\beta_{n+1}^2$, which is sent to $0$. It is straightforward to show that there is a surjective unital $K$-algebra homomorphism $\pi_n \colon Q_K(E_{n+1},C^{n+1}) \to Q_K(E_n,C^n)$ with these assignments.

    Set $Q:= \varprojlim Q_K(E_n,C^n)$. By Theorems \ref{T:Qreg} and \ref{T:surjective}, $Q$ is a unital regular $K$-algebra. However, $V(Q) \not\cong \varprojlim V(Q_K(E_n,C^n))$.
\end{example}
\thmskip

\noindent \emph{Proof.}
   To ease the notation, write $Q_n :=Q_K(E_n,C^n)$.  Our first, and most major, step is to show that the natural monoid homomorphism $\eta : V(Q) \rightarrow \varprojlim V(Q_n)$ is not injective.  To approach this, consider the sequences $e= (v,v,v,\dots)$ and $f = (v+w_0,v+w_0,v+w_0,\dots )$ in $Q$.   Observe that
  $$\eta([e]) = ([v],[v],\dots ) = ([v+w_0],[v+w_0],\dots ) = \eta([f])$$
  in $\varprojlim V(Q_n)$. This is due to the fact that for each $n$ we have
  $$v= \alpha_{i} \alpha_i^*+\beta_i^1(\beta_i^1)^* \sim v+w_0 $$
  in $Q_n$, for all $i>n$, since $\{\alpha_i,\beta_i^1\} \in C^n$ for $i>n$.

{\bf Claim 1:} $[e] \ne [f]$ in $V(Q)$, that is, $e$ is not equivalent to $f$ in $Q$.

  By way of contradiction, suppose that there are elements $x,y\in Q$ such that $x\in eQf$, $y\in fQe$ and $xy= e$, $yx= f$.

  Write $x= (x_0,x_1, \dots)$ and $y=(y_0,y_1,\dots )$, where $x_n,y_n \in Q_n$ and  $\pi_{n}(x_{n+1})= x_n$, $\pi_n (y_{n+1})= y_n$ for all $n$.   There exists a natural number $N$ such that all terms  appearing in the expressions of $x_0$ and $y_0$ belong to the subalgebra $Q^N_0 := Q_K(E_0^N,C_N^0)$ of $Q_0$, where
$$C_N^0 := \{ \{\alpha _i,\beta_i^1\}: 1\le i\le N\}$$
and $E_0^N$ is the subgraph of $E_0$ generated by the edges in $C_N^0$. Note that for all $n\ge 0$,
  $$x_{n+1} = \tilde{x}_n+ z_n, \qquad y_{n+1} = \tilde{y}_n+ q_n,$$
  where $\tilde{x}_n$ and $\tilde{y}_n$ are the elements of $Q_{n+1}$ obtained by the same expressions as the elements $x_n$ and $y_n$, but seen in the algebra $Q_{n+1}$, and $z_n,q_n\in Q_{n+1}w_{n+1}Q_{n+1}$.
   Therefore each of $z_n,q_n$ are sums of terms of the form $\gamma g(t_i^{w_{n+1}}) \nu^*$, where $\gamma $ is a fractional $c$-path ending at $w_{n+1}$ and $\nu$ is a $c$-path ending at $w_{n+1}$, and $g(t_i^{w_{n+1}})\in L_{w_{n+1}}$.

  Proceeding recursively, we can write
  $$x_N= \hat{x}_0 + \sum_{i=1}^N z_i',\qquad y_N= \hat{y}_0 + \sum_{i=1}^N q_i',$$
  where $\hat{x_0}, \hat{y}_0$ are elements of $Q_N$ which belong to the subalgebra $Q_N^N:= Q_K(E_N^N,C_N^N)$, where
  $$C_N^N := \{ \{\alpha _j,\beta_j^1,\beta_j^2\}: 1\le j\le N\}$$
  and $E_N^N$ is the subgraph of $E_N$ generated by the edges in $C_N^N$, and $z_i',q_i'\in Q_N w_iQ_N$.

  For $f\in \Sigma$, we set $u_{N,f}= fv+\sum_{i=0}^N w_i \in Q_N$ and $v_{N,f}= f^{-1}v+\sum_{i=0}^N w_i\in Q_N$. Observe that $u_{N,f}v_{N,f} = v+\sum_{i=0}^N w_i = v_{N,f}u_{N,f}$, and $v+\sum_{i=0}^N w_i $ is the unit of the algebra $Q_N$.

  We need the following lemma.

  \begin{lemma}
  	\label{lem:invariant-by-conjugation}
  	For $f\in \Sigma$ we have
  	$$u_{N,f} Q_N^N v_{N,f} \subseteq Q_N^N.$$
  	  \end{lemma}

   \begin{proof}
   	First observe that using \eqref{eq:second}, \eqref{eq:third} and the defining relations of $Q_N$, we get, for a $c$-path $\nu = \alpha_i^t\beta_i^s$, $t\ge 0, s= 1,2$,
   	$$\nu^* f^{-1} = (\beta_i^s)^* (\alpha_i^*)^t f^{-1} \in \sum_{u=0}^\infty L_{r(\beta_i^s)} (\beta_i^s)^* (\alpha_i^*)^u \subseteq Q_N^N.$$
It is also clear that $f\lambda \in Q_N^N$ for a fractional $c$-path $\lambda$ in $Q_N^N$.

   	With this observation at hand, we only need to show that $f \mathfrak m f^{-1}\in Q_N^N$ for each fractional monomial $\mathfrak m $ at $v$. Such a fractional monomial is of the form
   	$$\mathfrak m = h^{-1} \prod _{i=1}^N \alpha_i^{k_i}(\alpha_i^*)^{l_i} $$
   	for $h\in \Sigma_N$ and $k_i,l_i\ge 0$. Hence it suffices to check that $f\alpha_i^*f^{-1} \in Q_N^N$
   	for $1\le i\le N$. By \eqref{eq:second}, we have
   	$$f\alpha_i^* f^{-1}  = \alpha _i^* + (f_0')^{-1} g\rho^* (v-\alpha_i\alpha_i^*),$$
   	where $f_0'\in K[\alpha_j: j\ne i]\cap \Sigma$, $\rho$ is a momomial in $\{ \alpha_j : j\ne i \}$, and  $g\in K[\alpha_j]$. Note that $(f_0')^{-1} \beta_i^s \in \beta_i^s L_{r(\beta_i^s)}$ because $f_0'\in K[\alpha_j: j\ne i]$. Hence using repeatedly relations (7) and (8) and the latter observation we obtain that
 \begin{align*}
 	(f_0')^{-1}g \rho ^* (v-\alpha_i \alpha_i^*) & = \sum _{s=1}^2   g(f_0')^{-1} \rho ^* \beta_i^s(\beta_i^s)^* \\
 	& \in  \sum _{s=1}^2   g(f_0')^{-1} \beta_i^sL_{r(\beta_i^s)}(\beta_i^s)^*\\
 	& \subseteq  \sum _{s=1}^2   g \beta_i^sL_{r(\beta_i^s)}(\beta_i^s)^* \\
 	& \subseteq \sum _{s=1}^2 \sum_{u=0}^{\infty}  \alpha_i^u \beta_i^sL_{r(\beta_i^s)}(\beta_i^s)^*\subseteq Q_N^N.
 \end{align*}
 It follows that $f \mathfrak m f^{-1} \in Q_N^N$, as desired.
   	    	     \end{proof}

  Returning to the proof of the Example, observe that there exists $f\in \Sigma$ such that for each $1\le i\le N$ we have
  $$z'_i = v_{f,N}  (\sum_j   \gamma_{ij} z_{ij}\nu_{ij}^*),\quad q_i' = v_{f,N} (\sum _j \lambda_{ij} q_{ij} \mu_{ij}^*),$$
  where $\gamma_{ij}, \nu_{ij},\lambda_{ij},\mu_{ij}$ are (possibly trivial) $c$-paths in $Q_N^N$ ending at $w_i$, and $z_{ij},q_{ij}\in L_{w_i}$, for all $1\le i\le N$.

  As we have seen in the proof of Lemma \ref{lem:invariant-by-conjugation}, $\nu^* v_{f,N}\in Q_N^N$ for every $c$-path $\nu$ in $Q_N^N$. Hence we obtain
  $$u_{f,N}z_i' v_{f,N} = \sum_j \gamma_{ij}z_{ij}(\nu_{ij}^* v_{f,N}) \in Q_N^N,\qquad u_{f,N}q_i' v_{f,N} = \sum_j \lambda_{ij}q_{ij}(\mu_{ij}^* v_{f,N}) \in Q_N^N $$
  for all $i\in \{1,\dots , N\}$.
    Since $\hat{x}_0,\hat{y}_0\in Q_N^N$, this together with Lemma \ref{lem:invariant-by-conjugation} gives that $u_{f,N} x_N v_{f,N}\in Q^N_N$
    and $u_{f,N}y_N v_{f,N} \in Q^N_N$. Now observe that
    $$(u_{f,N}x_N v_{f,N})(u_{f,N}y_N v_{f,N}) = u_{f,N} vv_{f,N} = v$$
    and similarly
  $$(u_{f,N}y_N v_{f,N})(u_{f,N}x_N v_{f,N}) = u_{f,N} (v+w_0)v_{f,N} = v+w_0.$$
  We thus obtain that $v\sim v+w_0$ in $Q_N^N$. However we know that
  $$V(Q_N^N) \cong M(E_N^N,C_N^N)= \langle v,w_0,w_1,\dots ,w_N\mid v= v+w_0+w_i, 1\le i\le N \rangle,$$ and thus $v\nsim v+w_0$ in $Q_N^N$. This contradiction shows that $e$ is not equivalent to $f$ in $Q$, proving Claim 1.
  \medskip

We next exhibit the structure of $V(Q)$, making use of the canonical isomorphisms $M(E_n,C^n) \to V(Q_n)$ given by Theorem \ref{T:Qreg}.  We treat these isomorphisms as identifications.

First, $V(Q_0) = \langle v, w_0 : v = v+w_0 \rangle$. For $n>0$, the monoid $V(Q_n)$ is generated by $v,w_0,\dots,w_n$ with relations $v = v+w_0+w_j$ for $1\le j\le n$ and $v = v+w_0$, so in fact
$$
V(Q_n) = \langle v,w_0,\dots,w_n : v = v+w_j\ \forall\; 0\le j\le n \rangle \quad \forall\; n \ge 0.
$$
Note that $V(Q_n)$ has an o-ideal $W_n := \sum_{j=0}^n \Z_{\ge 0}\,w_j$ such that
\begin{itemize}
\item $V(Q_n) = W_n \sqcup \Zpos\,v$;
\item $kv \ne lv$ for all distinct $k,l \in \Zpos$;
\item $v+w = v$ for all $w \in W_n$;
\item $W_n = \{ x \in V(Q_n) : x\ \text{is cancellative in}\ V(Q_n) \}$.
\end{itemize}

The homomorphisms $V(\pi_n) : V(Q_{n+1}) \to V(Q_n)$ send $v \mapsto v$ and $w_j \mapsto w_j$ for $j \le n$ while $w_{n+1} \mapsto 0$.  Consequently, the monoid $\Vhat := \varprojlim V(Q_n)$ has an o-ideal $W := \varprojlim W_n$ and an element $\vhat := (v,v,\dots)$ such that
\begin{itemize}
\item $\Vhat = W \sqcup \Zpos\,\vhat$;
\item $\vhat + w = \vhat$ for all $w \in W$;
\item $W = \{ x \in \Vhat : x\ \text{is cancellative in}\ \Vhat \}$.
\end{itemize}
In particular, $\Vhat \setminus W$ is \emph{cyclic} in the sense that every element of this set is a positive multiple of the single element $\vhat$.

To get corresponding information about $V(Q)$, we need some of the ideal theory of $Q$.  Observe that for each $n \ge 0$ we have row-exact commutative $K$-algebra diagrams
\[
\xymatrixcolsep{3pc}
\xymatrix{
	0 \ar[r]  &	I_{n+1} \ar[r]^-{\subseteq} \ar[d] & Q_{n+1}  \ar[r]^-{\rho_{n+1}} \ar[d]^{\pi_n} & L \ar[r]\ar[d]^{\text{id}_L}  & 0 \\
	0 \ar[r] &	I_{n} \ar[r]^-{\subseteq}  & Q_{n}   \ar[r]^-{\rho_n} & L \ar[r]	& 0
}
\]
 where $I_n$ is the ideal of $Q_n$ generated by $w_0,w_1,\dots,w_n$ and $L$ is the rational function field $K(x_i : 1 \le i < \infty)$. It follows immediately that there is a surjective $K$-algebra homomorphism $\rho\colon Q\to L$ defined by
 $\rho ((q_0,q_1,\dots)) = \rho_n (q_n)$, which is independent of $n$. The kernel of $\rho$
is the ideal $I:= \varprojlim I_n$. Now observe that  $\{I_n,\pi_n\}$ is a surjective inverse system of (non-unital) semisimple rings.  It follows from Theorem \ref{T:surjective}  that $pQp$ is unit-regular for each idempotent $p = (p_0,p_1,\dots) \in I$, and from Theorem \ref{Thm:VInvLimit} that $V(I)\cong \varprojlim V(I_i) \cong \prod_{\Z_{\ge 0}} \Z_{\ge 0}$.

Thus $V(I)$ is a maximal o-ideal of $V(Q)$, and all elements in $V(I)$ are cancellative in $V(Q)$.  The idempotent $e \in Q$ is not in $I$, and $[e]$ is not cancellative in $V(Q)$, e.g. because $[e] = [e] + [f] + [g]$ where $g := (0,w_1,w_1,\dots)$.  Since the set of cancellative elements of $V(Q)$ is an o-ideal, it follows from the maximality of $V(I)$ that
\begin{itemize}
\item $V(I) = \{ x \in V(Q) : x\ \text{is cancellative in}\ V(Q) \}$.
\end{itemize}

Finally, suppose that $V(Q) \setminus V(I)$ is cyclic, say equal to $\Zpos\,u$ for some element $u$.  Then $[e] = ku$ and $[f] = lu$ for some $k,l \in \Zpos$.  Applying $V(\pi_0)$, we find that $v= ku_0$ and $v = v+w_0 = lu_0$ for some $u_0 \in V(Q_0)$.  From the structure of $V(Q_0)$, it follows that $k=l=1$.  But then $[e] = [f]$, contradicting Claim 1.  Thus $V(Q) \setminus V(I)$ is not cyclic.

Therefore $\Vhat \not\cong V(Q)$.
 \hfill $\square$
 \prskip

 \begin{remark}
 The Example \ref{E:Pere} shows that the lifting of units hypothesis in Proposition \ref{Prop:MatrixUnitCornerLifting} is not superfluous because the map $\eta$ in the Example is not injective. Curiously, it can be shown however that $\eta$ is surjective.
 \end{remark}
 \prskip

By Theorem \ref{T:Qreg}, each $Q_i$ is strongly separative. Note however, that the $Q_i$ can't be unit-regular, otherwise by Theorem \ref{Thm:VInvLimit} the map $\eta$ would be a monoid isomorphism.
\prskip

Since the connecting maps in the inverse limit of the Example are surjective, each $Q_i$ is a homomorphic image of $Q$. So having the $Q_i$ separative (resp.\ strongly separative) is necessary for the algebra $Q$ to be separative (resp.\ strongly separative). In fact, $Q$ is strongly separative. As shown in the proof above, $Q/I$ is a field and $V(I)$ is cancellative.
It follows from \cite[Theorem 5.5]{AGOP1} that $Q$ is strongly separative.
\prskip

Thus $Q$ is unit-regular by unit-regular, although $Q$ can't be unit-regular, since its quotients $Q_i$ are not unit-regular. Still, despite the complexity of the construction of Example \ref{E:Pere}, the resulting algebra $Q$ is ring-structurally still fairly nice. Of course, this raises the question of what additional complexities might one impose on separated graph algebras to get a nonseparative algebra via a construction similar to our Example? When one question ends, another begins --- a recurring pattern in mathematics!
\thmskip


\section{An intermediate step}\label{Section:FinalThoughts}
\thmskip

A \textbf{strongly separative} regular ring $R$ satisfies the property
\[
    A \ \oplus \ A \  \cong \  A \ \oplus \ B \  \ \Longrightarrow \ A \ \cong \ B
\]
for f.g. projective $R$-modules $A, B$, and more generally but equivalently the cancellation
\[
A \ \oplus \ C \  \cong \  B \ \oplus \ C \ \Longrightarrow \ A \ \cong \ B
\]
for f.g.\ projective $R$-modules $A,B,C$ when $C$ is isomorphic to a direct summand of a finite direct sum of copies of $A$. See \cite[Lemma 5.1]{AGOP1}. These rings are necessarily separative but not conversely because strong separativity implies the stable rank of $R$ is 1 or 2. See \cite[Theorem 3.3]{AGOP1}.
\prskip

\begin{question}\label{SSP}
Can we construct an inverse limit $\varprojlim R_i$ of strongly separative regular (resp. exchange) rings $R_i$ that is regular (resp. exchange) but not strongly separative?
\end{question}
\prskip

A positive answer to this question could well show the way to answering the Separativity Problem itself. In fact, it could even be that Question \ref{SSP} and its separative regular counterpart (in terms of inverse limits) will stand or fall together (both positive answers or both negative answers). Since a negative answer involving an inverse limit $\varprojlim R_i$ of strongly separative regular rings $R_i$ is a subring of $\prod R_i$, with the latter a strongly separative regular ring, at a minimum we would need to have an explicit example of a regular subring $R$ of a strongly separative regular $Q$ such that $R$ is not strongly separative (but almost certainly separative). Notice that in the ring $Q$ used in 4.3, \emph{all} regular subrings $R$ of $Q$ are strongly separative, because $Q$ is prime with nonzero socle and is a field modulo its socle. Therefore, $R$ will be an extension of its socle by a field, both strongly separative, whence so is $S$ by the Extension Theorem for strongly separative regular rings. On the other hand, the regular ring used in Construction 4.7 won't do as a suitable $Q$ because it is not strongly separative (for all corners of regular strongly separative rings have stable rank at most 2). Without the ``explicit'' requirement, we can achieve this initial minimum goal, but so far not otherwise. However, we can construct an explicit, non-strongly-separative exchange subring $R$ of a very nice unit-regular ring $Q$. Here are the details.
\prskip

Let $F$ be any field.  Let $B(F)$ be the ring of $\omega\times \omega$ row-and-column-finite matrices over $F$.  This ring is known to be an exchange ring by \cite[Theorem 1]{O1}; however, it is never regular.  Following the work done in \cite[p.\ 413--414]{GMM}, which follows earlier work of Tyukavkin, let
\[
R:=\biggl\{x=(x_n)_{n\in \Z_{>0}}\in \prod_{n\in \Z_{>0}}M_n(F)\, :\, \text{\emph{the rows and columns of the $x_n$ stabilize}}\biggr\}.
\]
There is a beautiful surjective algebra homomorphism of $R$ onto $B(F)$ given by the rule
\[
 t : R \longrightarrow B(F), \ \ \ x \ \longmapsto \ \ \mbox{\emph{matrix of stabilized rows and columns}}.
\]
Further, as explained in \cite{GMM}, the kernel of this Tyukavkin map $t$  is a (huge) unit-regular ideal.  Since unit-regular ideals are (nonunital) exchange rings, relative to which idempotents lift, and $B(F)$ is an exchange ring, $R$ is also an exchange ring.
\prskip

Let $B = B(F)$. It is straightforward to show $B \, \oplus B \, \cong B \, \cong B \, \oplus \ 0$, whence $B$ is not strongly separative. Therefore, $R$ is not  strongly separative either because it has $B$ as a homomorphic image. From \cite[Corollary 1.9]{APP} we know $B$ is separative and therefore so is $R$ by the Extension Theorem for separative regular rings. Finally, $R$ sits explicitly inside the unit-regular algebra $Q = \prod_{n\in \Z_{>0}}M_n(F)$. Initial minimum goal achieved at least for exchange rings.
\thmskip
\prskip

Our constructions in Section 4, showing failure of various regular ring properties to be preserved in inverse limits, used variations of a Bergman argument involving the fixed ring of a set of automorphisms. This method will not work, however, in showing for example the non-strongly separative $R$ above can be obtained as an inverse limit of strongly separative regular rings  --- for the simple reason that $R$ can't be the fixed ring of a set of automorphisms of some  subalgebra $S$ of $Q$ containing $R$. The argument is as follows. Let $I = \mbox{soc}(Q) = \bigoplus_{n=1}^\infty M_n(F)$. Note $\mbox{soc}(R) = \mbox{soc}(S) = I$ as well. Therefore an automorphism $\theta$ of $S$ induces an automorphism of $I$ and of each of the homogeneous components $M_n(F)$ when identified with $e_nS$ for the central idempotent $e_n = (0,0, \ldots, 1,0, \ldots \ )$ \,(because the homogeneous components have different lengths, they are not non-trivially permuted). In particular, $\theta(e_n) = e_n$. Hence the action of $\theta$ is completely determined by its restrictions to the $e_nS$, since the components of $\theta(s)$, for $s \in S$, are given by $\theta(s)(e_n) = \theta(se_n)$, thus determining $\theta(s)$. Therefore, if $\theta$ is a nontrivial automorphism of $S$, its fixed ring when restricted to some $e_kS$ must be a proper subalgebra of $e_kS$. But now the fixed ring of $\theta$ can't contain $I$. In particular, it can't fix $R$. So we need to develop other techniques.
\prskip

\emph{In summary, perhaps our paper has at least established a base camp from which our Mt.\ Everest, the Separativity Problem, might be conquered --- but watch out for crevasses.}
\prskip

\section*{Acknowledgements}
\prskip

{\small{We thank Pace Nielsen for his helpful suggestions. Pere Ara and Francesc Perera were partially supported by the Spanish State Research Agency (grant No.\ PID2020-113047GB-I00/AEI/10.13039/501100011033), by the Comissionat per Universitats i Recerca de la Generalitat de Catalunya (grant No.\ 2017-SGR-1725) and by the Spanish State Research Agency through the Severo Ochoa and Mar\'ia de Maeztu Program for Centers and Units of Excellence in R\&D (CEX2020-001084-M).  The research of Ken Goodearl was partially supported by US National Science Foundation grant DMS 1601184.  Kevin O'Meara thanks the Mathematics Departments of the University of California, Santa Barbara (USA), and of both the Universitat Autonoma de Barcelona  and Universidad de C\'adiz (Spain) for their hospitality during the early part of our project. Enrique Pardo was partially supported by PAI III grant FQM-298 of the Junta de Andaluc\'ia, by the DGI-MINECO and European Regional Development Fund, jointly, through grant PID2020-113047GB-I00, and by the grant ``Operator Theory: an interdisciplinary approach'', reference ProyExcel 00780, a project financed in the 2021 call for Grants for Excellence Projects, under a competitive bidding regime, aimed at entities qualified as Agents of the Andalusian Knowledge System, in the scope of the Plan Andaluz de Investigaci\'on, Desarrollo e Innovaci\'on (PAIDI 2020), Consejer\'ia de Universidad, Investigaci\'on e Innovaci\'on of the Junta de Andaluc\'ia.}}

\prskip

\bibliographystyle{amsplain}

\end{document}